\documentclass[reqno]{ws-jktr}

\usepackage{graphicx}
\usepackage{amscd}
\usepackage{amsmath, amssymb}
\usepackage{graphics}
\usepackage{graphicx}
\usepackage{caption}
\usepackage{subcaption}
\usepackage{epsfig}
\usepackage{xcolor}
\usepackage{soul}
\usepackage{tikz}
\usepackage{tcolorbox,amsmath,parcolumns}

 \newtheorem{question}[theorem]{Question}

\usepackage{xcolor}
\usepackage{amsmath}
\usepackage{tcolorbox}
\usepackage{multicol}

\title{ON REPRESENTATIONS OF THE MULTI-VIRTUAL BRAID GROUP $M_kVB_n$ AND THE MULTI-WELDED BRAID GROUP $M_kWB_n$}
\author{Vaibhav Keshari}
\address{Department of Mathematics\\
Indian Institute of Technology Ropar, Punjab,
India\\ vaibhav.23maz0022@iitrpr.ac.in}

\author{Mohamad N. Nasser}
\address{Department of Mathematics and Computer Science\\ 
         Beirut Arab University\\ m.nasser@bau.edu.lb}
         
\author{Madeti Prabhakar\footnote{Corresponding author}}
\address{Department of Mathematics\\
Indian Institute of Technology Ropar, Punjab,
India\\ prabhakar@iitrpr.ac.in}

\begin{document}
\maketitle
\begin{abstract}
    This paper classifies complex homogeneous $2$-local representations of the multiple virtual braid group $M_kVB_n$ into $\mathrm{GL}_n(\mathbb{C})$ for $n\geq3$ and $k >1$, showing that such representations fall into exactly $2^{k+1}+1$ distinct types, out of which except three all are unfaithful. In addition, this paper investigates complex homogeneous $2$-local representations of the multiple welded braid group $M_kWB_n$ into $\mathrm{GL}_n(\mathbb{C})$ for $n\geq3$ and $k >1$, identifying $3 \cdot 2^{k-1} +1$ representations. Moreover, the article includes a construction of a non-local representation of $M_2WB_3$ that extends the known LKB representation of the braid group on $3$ strands, namely $B_3$, making a path towards constructing non-local representations of $M_kWB_n$ in general.
\end{abstract}

\section{Introduction}
In 2004, L.~Kauffman~\cite{KauffmanLambropoulou2004} introduced the \emph{virtual braid group} $VB_n$, a natural generalization of the classical braid group $B_n$ that incorporates virtual crossings. This group is defined abstractly by the Artin generators $\sigma_1, \sigma_2, \ldots, \sigma_{n-1}$ (representing classical crossings) and $\rho_1, \rho_2, \ldots, \rho_{n-1}$ (representing virtual crossings), subject to a specific set of defining relations. Building upon virtual knot theory, \emph{welded knot theory} is defined by including an additional move known as the \emph{forbidden move} F1 (also called the \emph{over-forbidden move}). In this framework, L. Kauffman~\cite{LH2024} introduced also the \emph{multi-virtual braid group}, a further generalization of the virtual braid group that incorporates multiple families of virtual generators. These groups correspond to \emph{multi-virtual knot theory}, where various types of virtual crossings are permitted. The structure of the multi-virtual braid group is governed by \emph{detour moves}, which are allowed between certain types of crossings while explicitly forbidden between others. Furthermore, \emph{multi-welded knot theory} extends this setup by allowing the forbidden move F1 in the multi-virtual setting, leading to the definition of the \emph{multi-welded braid group}.

\vspace{0.1cm}

Understanding such generalized braid groups raises an important question: \textit{Are these groups linear?} In algebra, linearity of an algebraic structure refers to the existence of a faithful representation into a matrix group over a field. For braid groups, this property has played a crucial role in understanding their algebraic and geometric behavior. A landmark result in this context is the linearity of the classical braid group \( B_n \), established via Lawrence–Krammer–Bigelow (LKB) representation, which provides a faithful linear representation of $B_n$ into \( \text{GL}_\frac{n(n-1)}{2}(R) \) for $n\geq 2$ and a suitable ring \( R \) \cite{Big2001, Kram2002, Law1990}.
The irreducibility of a representation is another specialty in representation theory that aids in understanding any algebraic structure. Classifying irreducible representations of any group $G$ is an intriguing task, particularly for the braid group $B_n$ and its group and monoid extensions.

\vspace{0.1cm}

In~\cite{Mik2013}, Y. Mikhalchishina introduced the concept of \textit{local linear representations} of \( B_3 \) and further developed \textit{homogeneous $2$-local representations} of \( B_n \), for \( n \geq 3 \), highlighting the significance of locality in constructing matrix representations. Moreove, T. Mayassi \textit{et al.} generalized Mikhalchishina's work by classifying all \textit{homogeneous $3$-local representations} of $B_n$ for $n\geq 4$ \cite{Mayassi2025}. Additionally, M.~Nasser \textit{et al.}~\cite{MNas2025} determined {\it homogeneous $k$-local representations} of the flat virtual braid group $FVB_n$ for different degrees $k$. Building on this, M.~Nasser \textit{et al.}~\cite{Mnas2025} defined {\it homogeneous $2$-local representations} of the twisted virtual braid group \( TVB_n, n\geq 2 \), studying their faithfulness and irreducibility in some cases.

\vspace{0.1cm}

Motivated by these developments, it is natural to explore the algebraic structure of more generalized braid groups. In particular, the \textit{multi-virtual braid group} \( M_kVB_n \) generalizes the virtual braid group by introducing multiple types of virtual crossings, while the \textit{multi -welded braid group} \( M_kWB_n \) further includes the over-forbidden move (F1), extending the structure analogously to welded braid theory.

\vspace{0.1cm}

This paper aims to understand the structure of these generalized braid groups by constructing \textit{complex homogeneous $2$-local representations} of \( M_kVB_n \) and \( M_kWB_n \), and studying key representation-theoretic properties such as \textit{faithfulness} and \textit{irreducibility}. 
We begin in Section~2 by presenting the key definitions and propositions that form the foundation for the results developed throughout this paper. In Section~3, we investigate all complex homogeneous $2$-local representations of the multi-virtual braid group \( M_kVB_n \) and analyze their properties. Section~4 focuses on the complex homogeneous $2$-local representations of the multi-welded braid group \( M_kWB_n \). In Section~5, we construct a non-local representation of $M_2WB_3$ that extends the LKB representation of $B_3$ and we state some of its characteristics.

\section{Preliminary}
The braid group on $n$ strands $B_n$~\cite{artin1946}, $n\geq2$, is an abstract group defined by the Artin generators $\sigma_1, \sigma_2, \ldots, \sigma_{n-1}$ with the following defining relations.
\begin{align}
    \sigma_i\sigma_{i+1}\sigma_i & =\sigma_{i+1}\sigma_i \sigma_{i+1}, \hspace{1 cm} i=1,2,\ldots ,n-2, \\
 \sigma_i\sigma_j & =\sigma_j\sigma_i,\ \hspace{1.89 cm} |i-j|\geq2.
\end{align}
The generator $\sigma_i$ and its inverse $\sigma_i^{-1}$ are geometrically illustrated in Fig.~\ref{fig:cl2}.
\begin{figure}[h]
\centering
    \begin{subfigure}[c]{0.35\textwidth}
        \centering
     \includegraphics[width=0.8\textwidth]{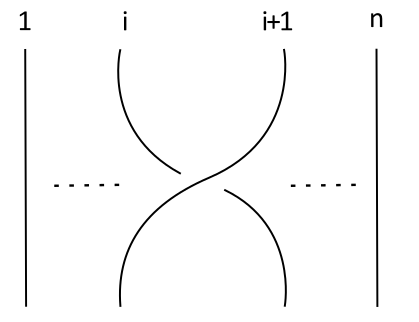}
     \caption{$\sigma_i$}
     \label{fig:cl1}
 \end{subfigure}
 ~
\begin{subfigure}[c]{0.35\textwidth}
            \centering
            \includegraphics[width=0.8\textwidth]{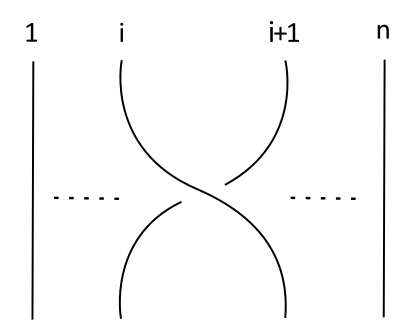}
         \caption{$\sigma^{-1}_i$} 
    \end{subfigure}
    \caption{~Geometrical interpretation of $\sigma_i$ and $\sigma^{-1}_i${}}
    \label{fig:cl2}
\end{figure}

The virtual braid group on $n$ strands $VB_n$ \cite{KauffmanLambropoulou2004}, $n\geq 2$, extends the classical braid group $B_n$ by adding generators $\rho_1,\rho_2,\dots, \rho_{n-1}$, along with the relations (2.1)-(2.2) and the additional defining relations as follows.
\begin{align}
    \rho^2_i &=e,  \hspace{2.47 cm}  i=1,2,\ldots,n-1, 
    \\\rho_i \rho_j &=\rho_j \rho_i, \hspace{1.95 cm}  |i-j|\geq 2,
  \\\rho_i \rho_{i+1} \rho_i &=\rho_{i+1} \rho_i \rho_{i+1}, \hspace{0.975 cm} i=1,2,\ldots, n-2, 
  \\\sigma_i \rho_j &=\rho_j \sigma_i, \hspace{1.93 cm} |i-j|\geq 2,
  \\\rho_i \rho_{i+1}\sigma_i &=\sigma_{i+1}\rho_i\rho_{i+1}, \hspace{0.95 cm}  i=1,2,\ldots, n-2.
\end{align}
The geometric interpretation of the identity $e$ and $\rho_i$ is shown in Fig.~\ref{fig:vir}.
\begin{figure}[h]
\centering
    \begin{subfigure}[c]{0.35\textwidth}
        \centering
     \includegraphics[width=0.9\textwidth]{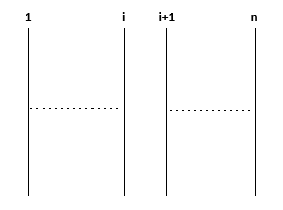}
     \caption{$e$}
     \label{fig:e}
 \end{subfigure}
 ~
\begin{subfigure}[c]{0.35\textwidth}
            \centering
            \includegraphics[width=0.8\textwidth]{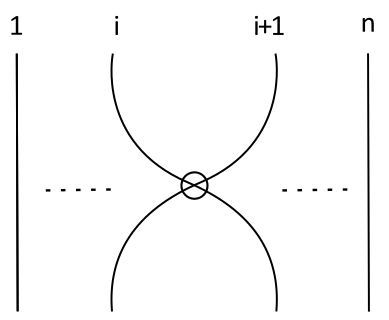}
         \caption{$\rho_i$} 
    \end{subfigure}
    \caption{~~Geometrical interpretation of identity e and $\rho_i$} 
    \label{fig:vir}
\end{figure}

The multi-virtual braid group on $n$ strands $M_kVB_n$~\cite{LH2024}, $n\geq 2$ and $k>1$, is an abstract group generated by $\sigma_1, \sigma_2, \dots, \sigma_{n-1}$ and $\rho_1^{\alpha}, \rho_2^{\alpha},\dots ,\rho_{n-1}^{\alpha}$ for $\alpha=0,1,\dots, k-1$, subject to a set of defining relations: 
\begin{align}
 \sigma_i\sigma_{i+1}\sigma_i & =\sigma_{i+1}\sigma_i \sigma_{i+1}, \hspace{1 cm} i=1,2,\ldots,n-2, \\
 \sigma_i\sigma_j & =\sigma_j\sigma_i,\ \hspace{1.9 cm} |i-j|\geq 2,\\
 (\rho_i^{\alpha})^2 & = e, \hspace{2.55cm} i=1,2,\ldots, n-1   \text{ and }  \alpha =0,1,\ldots,k-1,\\
 \rho_i^{\alpha} \rho_j^{\alpha} & = \rho_j^{\alpha} \rho_i^{\alpha},\hspace{1.9cm}   |i-j|\geq 2 \ \text{ and }  \alpha =0,1,\ldots, k-1,\\
 \rho_i^{\alpha} \rho_{i+1}^{\alpha} \rho_i^{\alpha} & = \rho_{i+1}^{\alpha} \rho_i^{\alpha} \rho_{i+1}^{\alpha}, \hspace{1cm} i=1,2,\ldots, n-1 \text{ and } \alpha =0,1,\ldots, k-1,\\
 \sigma_i \rho_j^{\alpha} & = \rho_j^{\alpha} \sigma_i, \hspace{1.98cm} |i-j|\geq2 \text{ and }  \alpha =0,1,\ldots,k-1,\\
 \rho_i^{\alpha} \rho_j^{\beta} & = \rho_j^{\beta} \rho_i^{\alpha}, \hspace{1.96cm} |i-j|\geq2 \text{ and } \alpha \neq \beta, \\
 \sigma_i \rho_{i+1}^{0} \rho_i^{0} & = \rho_{i+1}^{0} \rho_i^{0} \sigma_{i+1}, \hspace{1.05cm} i=1,2,\ldots,n-2,\\
 \rho_i^{0} \rho_{i+1}^{0} \rho_i^{\beta} & = \rho_{i+1}^{\beta} \rho_i^{0} \rho_{i+1}^{0}, \hspace{1.05cm} i=1,2,\ldots,n-2 \text{ and } \beta=1,2,\ldots, k-1. \vspace{0.1cm}
 \end{align}
 
 
 \noindent Note that for \( k = 1 \), \( M_kVB_n \) coincides with the virtual braid group \( VB_n \). When \( k = 2 \), \( M_kVB_n \) includes two distinct types of virtual crossings, as illustrated in Fig.~~\ref{fig:difvir}. In general, the generator \( \rho^\alpha_i \) is geometrically represented as shown in Fig.~~\ref{fig:mvir}. Varying \( \alpha \in \{0,1, \dots, k-1\} \) yields different types of virtual crossings.
 \begin{figure}[h]
\centering
    \begin{subfigure}[c]{0.35\textwidth}
        \centering
     \includegraphics[width=0.8\textwidth]{vir.png}
     \label{fig:vir0}
 \end{subfigure}
 ~
  \begin{subfigure}[c]{0.35\textwidth}
        \centering
     \includegraphics[width=0.8\textwidth]{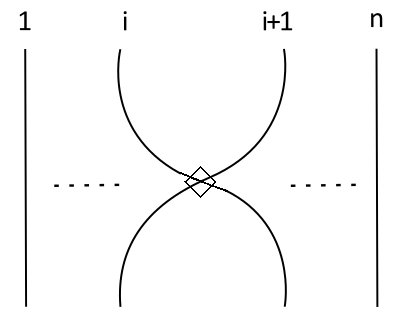}
     \label{fig:svir}
 \end{subfigure}
 \caption{Different type of Virtual crossing}
 \label{fig:difvir}
 \vspace{0.5cm}
 ~
\begin{subfigure}[c]{0.35\textwidth}
            \centering
            \includegraphics[width=0.8\textwidth]{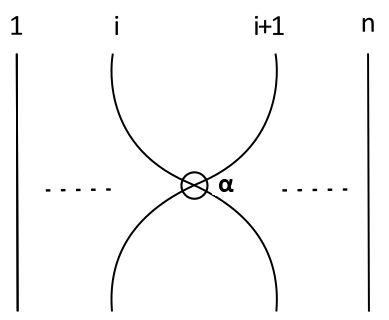}
         \caption{$\rho^{\alpha}_i$} 
    \end{subfigure}
    
    \caption{~~Geometrical interpretation of $\rho_i^{\alpha}$} 
    \label{fig:mvir}
\end{figure}\\
Remark that, for $1\leq i\leq n-2,$ the relations below are forbidden in $M_kVB_n$.
 \begin{align}
     F1:\hspace{0.1cm} \sigma_i \sigma_{i+1} \rho_i^{\alpha} & = \rho_{i+1}^{\alpha} \sigma_i \sigma_{i+1},  \hspace{0.86cm} 0 \leq \alpha \leq k-1,\\
     F2:\hspace{0.1cm} \sigma_{i+1} \sigma_i \rho_{i+1}^{\alpha} & = \rho_i^{\alpha} \sigma_{i+1} \sigma_i, \hspace{1.025cm} \hspace{0.1cm} 0 \leq \alpha \leq k-1,\\
      F3: \hspace{0.1cm} \rho_i^{(0)} \, \rho_{i+1}^{(\beta)} \, \rho_i^{(\beta)} &= \rho_{i+1}^{(\beta)} \, \rho_i^{(\beta)} \, \rho_{i+1}^{(0)}, \hspace{0.5cm}  0 < \beta \leq k-1, \\
\rho_i^{(\gamma)} \, \rho_{i+1}^{(\beta)} \, \rho_i^{(\beta)} &= \rho_{i+1}^{(\beta)} \, \rho_i^{(\beta)} \, \rho_{i+1}^{(\gamma)},  \hspace{0.5cm}  0 < \gamma < \beta \leq k-1, \\
\rho_i^{(\gamma)} \, \rho_{i+1}^{(\gamma)} \, \rho_i^{(\beta)} &= \rho_{i+1}^{(\beta)} \, \rho_i^{(\gamma)} \, \rho_{i+1}^{(\gamma)}, \hspace{0.5cm} 0 < \gamma < \beta \leq k-1.
 \end{align}
 
 The multi-welded braid group on $n$ strands $M_kWB_n$ \cite{up}, $n\geq 2$ and $k>1$, is the group obtained from $M_kVB_n$ by allowing the first forbidden move (F1) and the multi-unrestricted braid group $M_kUB_n$ \cite{up}, $n \geq 2$ and $k>1$, is the group obtained from $M_kVB_n$ by allowing the first and second forbidden moves (F1 and F2).




\vspace{0.1cm}
We now introduce the concept of $k$-local representations of any group with finite number of generators.

\begin{definition}
Let $t$ be indeterminate and let $G$ be a group with generators $g_1,g_2,\ldots,g_{n-1}$. A representation \[ \theta: G \rightarrow \text{GL}_{m}(\mathbb{Z}[t^{\pm 1}]) \] is said to be $k$-local if it is of the form
$$\theta(g_i) =\left( \begin{array}{c|@{}c|c@{}}
   \begin{matrix}
     I_{i-1} 
   \end{matrix} 
      & \textbf{0} & \textbf{0} \\
      \hline
    \textbf{0} &\hspace{0.2cm} \begin{matrix}
   		M_i
   		\end{matrix}  & \textbf{0}  \\
\hline
\textbf{0} & \textbf{0} & I_{n-i-1}
\end{array} \right), \hspace*{0.2cm} \text{for} \hspace*{0.2cm} 1\leq i\leq n-1,$$ 
where $M_i \in \text{GL}_k(\mathbb{Z}[t^{\pm 1}])$ with $k=m-n+2$ and $I_r$ is the $r\times r$ identity matrix. The representation $\theta$ is said to be homogeneous if the matrices \(M_i, 1 \leq i \leq n-1\), are equal.
\end{definition}

As important examples of homogeneous $k$-local representations with different degrees $k$, we consider the following two local representations of the braid group $B_n$, the \emph{Burau representation}, which is homogeneous $2$-local and the \emph{$F$-representation}, which is homogeneous $3$-local.

\begin{definition}\cite{Bur1936} \label{defBurau}
For $t$ indeterminate and $n\geq 2$, the Burau representation \[\mathcal{B}: B_n\rightarrow \text{GL}_n(\mathbb{Z}[t^{\pm 1}])\] is the representation given as follows:
$$\mathcal{B}(\sigma_i)= \left( \begin{array}{c|@{}c|c@{}}
   \begin{matrix}
     I_{i-1} 
   \end{matrix} 
      & \textbf{0} & \textbf{0} \\
      \hline
    \textbf{0} &\hspace{0.2cm} \begin{matrix}
   	1-t & t\\
   	1 & 0\\
\end{matrix}  & \textbf{0}  \\
\hline
\textbf{0} & \textbf{0} & I_{n-i-1}
\end{array} \right), \hspace*{0.2cm} \text{for} \hspace*{0.2cm} 1\leq i\leq n-1.$$ 
\end{definition}

\begin{definition} \cite{19} \label{Fdef}
For $t$ be indeterminate and $n\geq 2$, the $F$-representation $$\mathcal{F}: B_n \rightarrow \text{GL}_{n+1}(\mathbb{Z}[t^{\pm 1}])$$ is the representation given as follows:
$$\mathcal{F}(\sigma_i) = \left( \begin{array}{c|@{}c|c@{}}
   \begin{matrix}
     I_{i-1} 
   \end{matrix} 
      & 0 & 0 \\
      \hline
    0 &\hspace{0.2cm} \begin{matrix}
   		1 & 1 & 0 \\
   		0 &  -t & 0 \\   		
   		0 &  t & 1 \\
   		\end{matrix}  & 0  \\
\hline
0 & 0 & I_{n-i-1}
\end{array} \right), \hspace*{0.2cm} \text{for} \hspace*{0.2cm} 1\leq i\leq n-1.$$ 
\end{definition}

\vspace{0.1cm}

The faithfulness of the Burau representation varies with $n$. 
For $n\leq 3$, it is proven to be faithful \cite{Bir1975}, while for $n\geq 5$, it is shown to be unfaithful \cite{Moo1991, Long1992, Big1999}. The specific case of $n=4$ is still unresolved. On the other hand, the question of faithfulness for the $F$-representation is still open for $n\geq 2$. Furthermore, concerning the irreducibility of the Burau and the $F$-representations, they have been established as reducible for $n\geq 3$ (see \cite{For1996} and \cite{55} respectively).

\vspace{0.1cm}

The notion of $k$-local representations can be extended to a group $G$ generated by $k(n-1)$ elements, where the generating set is partitioned into $k$ distinct families, each consisting of $n-1$ generators. For simplicity, we illustrate the construction in the case $k=2$.

\begin{definition}
Let $t$ be indeterminate and let $G$ be a group with two families of generators $g_1,g_2,\ldots,g_{n-1}$ and $h_1,h_2,\ldots,h_{n-1}$. A $k$-local representation \[ \theta: G \rightarrow \text{GL}_{m}(\mathbb{Z}[t^{\pm 1}]) \] is a representation of the form
$$\theta(g_i) =\left( \begin{array}{c|@{}c|c@{}}
   \begin{matrix}
     I_{i-1} 
   \end{matrix} 
      & \textbf{0} & \textbf{0} \\
      \hline
    \textbf{0} &\hspace{0.2cm} \begin{matrix}
   		M_i
   		\end{matrix}  & \textbf{0}  \\
\hline
\textbf{0} & \textbf{0} & I_{n-i-1}
\end{array} \right),$$ \vspace{0.25cm}  \text{ and \ } 
$$\theta(h_i) =\left( \begin{array}{c|@{}c|c@{}}
   \begin{matrix}
     I_{i-1} 
   \end{matrix} 
      & \textbf{0} & \textbf{0} \\
      \hline
    \textbf{0} &\hspace{0.2cm} \begin{matrix}
   		N_i
   		\end{matrix}  & \textbf{0}  \\
\hline
\textbf{0} & \textbf{0} & I_{n-i-1}
\end{array} \right), \vspace{0.5cm} $$
for $1\leq i\leq n-1,$ where $M_i,N_i \in \text{GL}_k(\mathbb{Z}[t^{\pm 1}])$ with $k=m-n+2$ and $I_r$ is the $r\times r$ identity matrix. In this case, $\theta$ is homogeneous if all the matrices $M_i$'s are equal and all the matrices $N_i$'s are equal.
\end{definition}

In the definition that follows, we generalize the notion of $k$-local representations to the multi-virtual braid group $M_kVB_n$. Recall that $M_kVB_n$ is generated by $k+1$ distinct families, the classical braid generators $\{\sigma_1, \sigma_2, \dots, \sigma_{n-1}\}$ and the virtual generators \{$\rho_i^{\alpha}$ $| 1\leq i\leq n-1, 0 \leq \alpha \leq k-1$ \}. We use the dummy variable $k'$ instead of $k$ in next definition to distinguish between the variables.

\begin{definition}
    Let $t$ be indeterminate. A $k'$-local representation \[ \theta: M_kVB_n \rightarrow \text{GL}_{m}(\mathbb{Z}[t^{\pm 1}]) \] is a representation that is defined on the families of generators of $M_kVB_n$ as follows.
$$\theta(\sigma_i) =\left( \begin{array}{c|@{}c|c@{}}
   \begin{matrix}
     I_{i-1} 
   \end{matrix} 
      & \textbf{0} & \textbf{0} \\
      \hline
    \textbf{0} &\hspace{0.2cm} \begin{matrix}
   		S_i
   		\end{matrix}  & \textbf{0}  \\
\hline
\textbf{0} & \textbf{0} & I_{n-i-1}
\end{array} \right),$$ \text{ and  } 
$$\theta(\rho_i^{\alpha}) =\left( \begin{array}{c|@{}c|c@{}}
   \begin{matrix}
     I_{i-1} 
   \end{matrix} 
      & \textbf{0} & \textbf{0} \\
      \hline
    \textbf{0} &\hspace{0.2cm} \begin{matrix}
   		R_i^{\alpha}
   		\end{matrix}  & \textbf{0}  \\
\hline
\textbf{0} & \textbf{0} & I_{n-i-1}
\end{array} \right),$$ 
for $1\leq i\leq n-1$ and  $0\leq \alpha \leq k-1,$ where $S_i,R_i^{\alpha} \in \text{GL}_{k'}(\mathbb{Z}[t^{\pm 1}])$ with $k'=m-n+2$, and $I_r$ is the $r\times r$ identity matrix. In this case, $\theta$ is homogeneous if all the matrices $S_i$'s are equal and all the matrices $R_i^{\alpha}$'s are equal for each fixed $\alpha$,  $0\leq \alpha \leq k-1$.
\end{definition}

We now recall some known results concerning 2-local representations of the braid group \( B_n \) and their key properties. These results will be instrumental in analyzing the irreducibility of the representations constructed in Section~3.

\begin{theorem} \cite{Mik2013} \label{ThMi}
Consider $n\geq 3$ and let $\beta: B_n \rightarrow \text{GL}_n(\mathbb{C})$ be a non-trivial complex homogeneous $2$-local representation of $B_n$. Then, $\beta$ is equivalent to one of the following three representations.
\begin{itemize}
\item[(1)] $\beta_1: B_n \rightarrow \text{GL}_n(\mathbb{C}) \hspace*{0.15cm} \text{such that } 
\beta_1(\sigma_i) =\left( \begin{array}{c|@{}c|c@{}}
   \begin{matrix}
     I_{i-1} 
   \end{matrix} 
      & 0 & 0 \\
      \hline
    0 &\hspace{0.2cm} \begin{matrix}
   		a & \frac{1-a}{c}\\
   		c & 0
   		\end{matrix}  & 0  \\
\hline
0 & 0 & I_{n-i-1}
\end{array} \right), \\ \text{ where } c \neq 0, a\neq 1, \hspace*{0.15cm} \text{for} \hspace*{0.2cm} 1\leq i\leq n-1.$ \\
\item[(2)] $\beta_2: B_n \rightarrow \text{GL}_n(\mathbb{C}) \hspace*{0.15cm} \text{such that }
\beta_2(\sigma_i) =\left( \begin{array}{c|@{}c|c@{}}
   \begin{matrix}
     I_{i-1} 
   \end{matrix} 
      & 0 & 0 \\
      \hline
    0 &\hspace{0.2cm} \begin{matrix}
   		0 & \frac{1-d}{c}\\
   		c & d
   		\end{matrix}  & 0  \\
\hline
0 & 0 & I_{n-i-1}
\end{array} \right),\\ \text{ where } c\neq 0, d\neq 1 \hspace*{0.2cm} \text{for} \hspace*{0.2cm} 1\leq i\leq n-1.$\\
\item[(3)] $\beta_3: B_n \rightarrow \text{GL}_n(\mathbb{C}) \hspace*{0.15cm} \text{such that } 
\beta_3(\sigma_i) =\left( \begin{array}{c|@{}c|c@{}}
   \begin{matrix}
     I_{i-1} 
   \end{matrix} 
      & 0 & 0 \\
      \hline
    0 &\hspace{0.2cm} \begin{matrix}
   		0 & b\\
   		c & 0
   		\end{matrix}  & 0  \\
\hline
0 & 0 & I_{n-i-1}
\end{array} \right),\\ \text{ where } bc\neq 0 \hspace*{0.2cm} \text{for} \hspace*{0.2cm} 1\leq i\leq n-1.$
\end{itemize}
\end{theorem}

\begin{theorem} \cite{M.Ch} \label{ThCh1}
The homogeneous $2$-local representations of the form $\beta_1$ and $\beta_2$ defined in Theorem \ref{ThMi} are reducible for $n\geq 6$.
\end{theorem}

\begin{theorem} \cite{M.Ch} \label{ThCh2}
The homogeneous $2$-local representations of the form $\beta_3$ defined in Theorem \ref{ThMi} are irreducible if and only if $bc\neq 1$, for $n\geq 3$.
\end{theorem}


It is natural to ask the following question: What are the possible forms of the complex homogeneous $2$-local representations of the multi-virtual braid group $M_kVB_n$ to $\text{GL}_n(\mathbb{C})$, and what are their key properties such as faithfulness and irreducibility?\\ In the next section, we  answer this question for $k=2$.

 \section{Complex Homogeneous Local Representation of $M_kVB_n$}

 In this section, we classify all complex homogeneous $2$-local representations of the multi-virtual braid group $M_kVB_n$ for $n\geq 3$ and $ k > 1$.

 \begin{theorem} \label{th3.1}
     Let $\beta: M_kVB_n \rightarrow \text{GL}_n(\mathbb{C})$ be a complex homogeneous $2$-local representation of $M_kVB_n$ for $n\geq 3$ and $ k > 1$. Then $\beta$ is equivalent to one of the following $2^{k+1}+1$ distinct representations. Among these, one is trivial representation in which $\beta$ sends each generator to the identity matrix, and the remaining $2^{k+1}$ representations are of the form:
     $$\beta(\sigma_i) =\left( \begin{array}{c|@{}c|c@{}}
   \begin{matrix}
     I_{i-1} 
   \end{matrix} 
      & \textbf{0} & \textbf{0} \\
      \hline
    \textbf{0} &\hspace{0.2cm} \begin{matrix}
   		S
   		\end{matrix}  & \textbf{0}  \\
\hline
\textbf{0} & \textbf{0} & I_{n-i-1}
\end{array} \right),$$ \vspace{0.25cm}
 and 
$$\beta(\rho_i^{\alpha}) =\left( \begin{array}{c|@{}c|c@{}}
   \begin{matrix}
     I_{i-1} 
   \end{matrix} 
      & \textbf{0} & \textbf{0} \\
      \hline
    \textbf{0} &\hspace{0.2cm} \begin{matrix}
   		X_{\alpha}
   		\end{matrix}  & \textbf{0}  \\
\hline
\textbf{0} & \textbf{0} & I_{n-i-1}
\end{array} \right), \vspace{0.25cm}$$
for $i=1,2,\dots,n-1,$ where \( S \) and \( X_{\alpha} \) are chosen from the following families of matrices:  
\[
S \in \left\{ 
\begin{pmatrix} 1 & 0 \\ 0 & 1 \end{pmatrix},\ 
\begin{pmatrix} 1 - bc & b \\ c & 0 \end{pmatrix},\ 
\begin{pmatrix} 0 & b \\ c & 0 \end{pmatrix},\ 
\begin{pmatrix} 0 & -\frac{d-1}{c} \\ c & d \end{pmatrix} 
\right\} \ \  \text{with } bc \neq 0,\ d \neq 1,
\]
and
\[
X_{\alpha} \in \left\{
\begin{pmatrix} 0 & x_{\alpha} \\ \frac{1}{x_{\alpha}} & 0 \end{pmatrix},\ 
\begin{pmatrix} 1 & 0 \\ 0 & 1 \end{pmatrix}
\right\} \ \  \text{for } 1 \leq \alpha \leq k-1\ \text{with } x_{\alpha} \neq 0.
\]
Additionally, for \(\alpha = 0\), we fix:
\[
X_0 = \begin{pmatrix} 0 & x_0 \\ \frac{1}{x_0} & 0 \end{pmatrix} \ \ \text{with } x_0 \neq 0.
\]
 \end{theorem}
 \begin{proof}
     We first prove it for $M_2VB_n$, since $M_2VB_n$ has only three family of generators $\sigma_i$'s, $\rho_i^{0}$'s, and $ \rho_i^{1}$'s. Set 
     $$\beta(\sigma_i) =\left( \begin{array}{c|@{}c|c@{}}
   \begin{matrix}
     I_{i-1} 
   \end{matrix} 
      & \textbf{0} & \textbf{0} \\
      \hline
    \textbf{0} &\hspace{0.2cm} \begin{matrix}
\begin{matrix}
a & b \\
c & d
\end{matrix}

   		\end{matrix}  & \textbf{0}  \\
\hline
\textbf{0} & \textbf{0} & I_{n-i-1}
\end{array} \right),$$ 
$$\beta(\rho_i^{0}) =\left( \begin{array}{c|@{}c|c@{}}
   \begin{matrix}
     I_{i-1} 
   \end{matrix} 
      & \textbf{0} & \textbf{0} \\
      \hline
    \textbf{0} &\hspace{0.2cm} \begin{matrix}
   		\begin{matrix}
w & x \\
y & z
\end{matrix}
   		\end{matrix}  & \textbf{0}  \\
\hline
\textbf{0} & \textbf{0} & I_{n-i-1}
\end{array} \right),$$ \text{and} \
$$\beta(\rho_i^{1}) =\left( \begin{array}{c|@{}c|c@{}}
   \begin{matrix}
     I_{i-1} 
   \end{matrix} 
      & \textbf{0} & \textbf{0} \\
      \hline
    \textbf{0} &\hspace{0.2cm} \begin{matrix}
   		\begin{matrix}
p & q \\
r & s
\end{matrix}
   		\end{matrix}  & \textbf{0}  \\
\hline
\textbf{0} & \textbf{0} & I_{n-i-1}
\end{array} \right),$$ 
where $a,b,c,d,w,x,y,z,p,q,r,s \in \mathbb{C}$, $ad-bc\neq 0,~ wz-xy\neq 0, ~ps-qr\neq 0$. We only need to consider the following relations of the generators of $M_2VB_n$, and the rest of the relations imply similar equations.
\begin{align}
         (\rho_1^{0})^2 & = e,  \\
         (\rho_1^{1})^2 & = e,\\
         \sigma_1 \sigma_2 \sigma_1 & = \sigma_2 \sigma_1 \sigma_2,\\
         \rho_1^{0} \rho_2^{0} \rho_1^{0} & = \rho_2^{0} \rho_1^{0} \rho_2^{0},\\
         \rho_1^{1} \rho_2^{1} \rho_1^{1} & = \rho_2^{1} \rho_1^{1} \rho_2^{1},\\
         \sigma_1 \rho_2^{0} \rho_1^{0} & = \rho_2^{0} \rho_1^{0} \sigma_2,\\
         \rho_1^{0} \rho_2^{0} \rho_1^{1} & = \rho_2^{1} \rho_1^{0} \rho_2^{0}.
    \end{align}
To determine all complex homogeneous $2$-local representations of $M_2VB_n$, we substitute the generator images under the representation $\beta$ into the above relations of the group. This substitution produces a system of $43$ algebraic equations involving $12$ unknowns.
\begin{tcolorbox}[
  title=System of Equations Corresponding to Relations (3.1)--(3.7),
  colframe=black,
  colback=white,
  sharp corners,
  fonttitle=\bfseries,
  ]

\renewcommand{\arraystretch}{1.2}  
\setlength{\tabcolsep}{12pt}    

\begin{tabular}{p{0.47\textwidth}|p{0.47\textwidth}}

\parbox[t]{\linewidth}{
\rule{\linewidth}{0.4pt}
\textbf{(i)} \quad \textbf{\boldmath$(\rho_1^{0})^2 = e$} \\
\rule{\linewidth}{0.4pt}
\vspace{-0.5\baselineskip}
\begin{align*}
w^2 + xy &= 1, && (3.8)\\
wx + xz &= 0, && (3.9)\\
wy + yz &= 0, && (3.10)\\
z^2 + xy &= 1, && (3.11)
\end{align*}
\rule{\linewidth}{0.4pt}

\textbf{(ii)} \quad \textbf{\boldmath$(\rho_1^{1})^2 = e$} \\
\rule{\linewidth}{0.4pt}
\vspace{-0.5\baselineskip}
\begin{align*}
p^2 + qr &= 1, && (3.12)\\
pq + qs &= 0, && (3.13)\\
pr + rs &= 0, && (3.14)\\
s^2 + qr &= 1, && (3.15)
\end{align*}
\rule{\linewidth}{0.4pt}

\textbf{(iii)} \quad \textbf{\boldmath$\sigma_1 \sigma_2 \sigma_1 = \sigma_2 \sigma_1 \sigma_2$} \\
\rule{\linewidth}{0.4pt}
\vspace{-0.5\baselineskip}
\begin{align*}
a^2 + bca &= a, && (3.16)\\
ab + abd &= ab, && (3.17)\\
ac + acd &= ca, && (3.18)\\
ad^2 + bc &= da^2 + bc, && (3.19)\\
bd + abd &= bd, && (3.20)\\
cd + acd &= cd, && (3.21)\\
d^2 + bcd &= d, && (3.22)
\end{align*}
\rule{\linewidth}{0.4pt}

\textbf{(iv)} \quad \textbf{\boldmath$\rho_1^{0} \rho_2^{0} \rho_1^{0} = \rho_2^{0} \rho_1^{0} \rho_2^{0}$} \\
\rule{\linewidth}{0.4pt}
\vspace{-0.5\baselineskip}
\begin{align*}
w^2 + xyw &= w, && (3.23)\\
wx + wxz &= wx, && (3.24)\\
wy + wyz &= wy, && (3.25)\\
wz^2 + xy &= zw^2 + xy, && (3.26)\\
xz + wxz &= xz, && (3.27)\\
yz + wyz &= yz, && (3.28)\\
z^2 + xyz &= z, && (3.29)
\end{align*}
}
&
\parbox[t]{\linewidth}{
\rule{\linewidth}{0.4pt}
\textbf{(v)} \quad \textbf{\boldmath$\rho_1^{1} \rho_2^{1} \rho_1^{1} = \rho_2^{1} \rho_1^{1} \rho_2^{1}$} \\
\rule{\linewidth}{0.4pt}
\vspace{-0.5\baselineskip}
\begin{align*}
p^2 + qrp &= p, && (3.30)\\
pq + pqs &= pq, && (3.31)\\
pr + prs &= pr, && (3.32)\\
ps^2 + qr &= sp^2 + qr, && (3.33)\\
qs + pqs &= qs, && (3.34)\\
rs + prs &= rs, && (3.35)\\
s^2 + qrs &= s, && (3.36)
\end{align*}
\rule{\linewidth}{0.4pt}

\textbf{(vi)} \quad \textbf{\boldmath$\sigma_1 \rho_2^{0} \rho_1^{0} = \rho_2^{0} \rho_1^{0} \sigma_2$} \\
\rule{\linewidth}{0.4pt}
\vspace{-0.5\baselineskip}
\begin{align*}
aw + bwy &= w, && (3.37)\\
ax + bwz &= ax, && (3.38)\\
cw + dwy &= wy, && (3.39)\\
cx + dwz &= cx + awz, && (3.40)\\
dx + bwz &= dx, && (3.41)\\
cz + ayz &= yz, && (3.42)\\
dz + byz &= z, && (3.43)
\end{align*}
\rule{\linewidth}{0.4pt}

\textbf{(vii)} \quad \textbf{\boldmath$\rho_1^{0} \rho_2^{0} \rho_1^{1} = \rho_2^{1} \rho_1^{0} \rho_2^{0}$} \\
\rule{\linewidth}{0.4pt}
\vspace{-0.5\baselineskip}
\begin{align*}
pw + rwx &= w, && (3.44)\\
qw + swx &= wx, && (3.45)\\
py + rwz &= py, && (3.46)\\
qy + swz &= qy + pwz, && (3.47)\\
qz + pxz &= xz, && (3.48)\\
sy + rwz &= sy, && (3.49)\\
sz + rxz &= z. && (3.50)
\end{align*}
\rule{\linewidth}{0.4pt}
}
\end{tabular}
\end{tcolorbox}

Solving the above equations (3.8)-(3.50) implies nine different solutions that give nine different representations of $M_2VB_n$ and they are:

\begin{itemize}
    \item [(1)] $\beta_1: M_2VB_n \rightarrow \text{GL}_n(\mathbb{C})$ such that
    $$\beta_1(\sigma_i) =\left( \begin{array}{c|@{}c|c@{}}
   \begin{matrix}
     I_{i-1} 
   \end{matrix} 
      & \textbf{0} & \textbf{0} \\
      \hline
    \textbf{0} &\hspace{0.2cm} \begin{matrix}
\begin{matrix}
1 & 0 \\
0 & 1
\end{matrix}

   		\end{matrix}  & \textbf{0}  \\
\hline
\textbf{0} & \textbf{0} & I_{n-i-1}
\end{array} \right), \
\beta_1(\rho_i^{0}) =\left( \begin{array}{c|@{}c|c@{}}
   \begin{matrix}
     I_{i-1} 
   \end{matrix} 
      & \textbf{0} & \textbf{0} \\
      \hline
    \textbf{0} &\hspace{0.2cm} \begin{matrix}
   		\begin{matrix}
1 & 0 \\
0 & 1
\end{matrix}
   		\end{matrix}  & \textbf{0}  \\
\hline
\textbf{0} & \textbf{0} & I_{n-i-1}
\end{array} \right),$$
and
$$\beta_1(\rho_i^{1}) =\left( \begin{array}{c|@{}c|c@{}}
   \begin{matrix}
     I_{i-1} 
   \end{matrix} 
      & \textbf{0} & \textbf{0} \\
      \hline
    \textbf{0} &\hspace{0.2cm} \begin{matrix}
   		\begin{matrix}
1 & 0 \\
0 & 1
\end{matrix}
   		\end{matrix}  & \textbf{0}  \\
\hline
\textbf{0} & \textbf{0} & I_{n-i-1}
\end{array} \right).$$ 

\item [(2)] $\beta_2: M_2VB_n \rightarrow \text{GL}_n(\mathbb{C})$ such that
    $$\beta_2(\sigma_i) =\left( \begin{array}{c|@{}c|c@{}}
   \begin{matrix}
     I_{i-1} 
   \end{matrix} 
      & \textbf{0} & \textbf{0} \\
      \hline
    \textbf{0} &\hspace{0.2cm} \begin{matrix}
\begin{matrix}
1-bc & b\\
c & 0
\end{matrix}

   		\end{matrix}  & \textbf{0}  \\
\hline
\textbf{0} & \textbf{0} & I_{n-i-1}
\end{array} \right), \
\beta_2(\rho_i^{0}) =\left( \begin{array}{c|@{}c|c@{}}
   \begin{matrix}
     I_{i-1} 
   \end{matrix} 
      & \textbf{0} & \textbf{0} \\
      \hline
    \textbf{0} &\hspace{0.2cm} \begin{matrix}
   		\begin{matrix}
0 & x \\
\frac{1}{x} & 0
\end{matrix}
   		\end{matrix}  & \textbf{0}  \\
\hline
\textbf{0} & \textbf{0} & I_{n-i-1}
\end{array} \right),$$
and
$$\beta_2(\rho_i^{1}) =\left( \begin{array}{c|@{}c|c@{}}
   \begin{matrix}
     I_{i-1} 
   \end{matrix} 
      & \textbf{0} & \textbf{0} \\
      \hline
    \textbf{0} &\hspace{0.2cm} \begin{matrix}
   		\begin{matrix}
1 & 0 \\
0 & 1
\end{matrix}
   		\end{matrix}  & \textbf{0}  \\
\hline
\textbf{0} & \textbf{0} & I_{n-i-1}
\end{array} \right),$$
where $b,c,x$ $\in \mathbb{C}$ , $bc\neq0$ and $x\neq0$.
   
\item [(3)] $\beta_3: M_2VB_n \rightarrow \text{GL}_n(\mathbb{C})$ such that
   $$\beta_3(\sigma_i) =\left( \begin{array}{c|@{}c|c@{}}
   \begin{matrix}
     I_{i-1} 
   \end{matrix} 
      & \textbf{0} & \textbf{0} \\
      \hline
    \textbf{0} &\hspace{0.2cm} \begin{matrix}
\begin{matrix}
0 & b\\
c & 0
\end{matrix}

   		\end{matrix}  & \textbf{0}  \\
\hline
\textbf{0} & \textbf{0} & I_{n-i-1}
\end{array} \right), \
\beta_3(\rho_i^{0}) =\left( \begin{array}{c|@{}c|c@{}}
   \begin{matrix}
     I_{i-1} 
   \end{matrix} 
      & \textbf{0} & \textbf{0} \\
      \hline
    \textbf{0} &\hspace{0.2cm} \begin{matrix}
   		\begin{matrix}
0 & x \\
\frac{1}{x} & 0
\end{matrix}
   		\end{matrix}  & \textbf{0}  \\
\hline
\textbf{0} & \textbf{0} & I_{n-i-1}
\end{array} \right),$$
and
$$\beta_3(\rho_i^{1}) =\left( \begin{array}{c|@{}c|c@{}}
   \begin{matrix}
     I_{i-1} 
   \end{matrix} 
      & \textbf{0} & \textbf{0} \\
      \hline
    \textbf{0} &\hspace{0.2cm} \begin{matrix}
   		\begin{matrix}
1 & 0 \\
0 & 1
\end{matrix}
   		\end{matrix}  & \textbf{0}  \\
\hline
\textbf{0} & \textbf{0} & I_{n-i-1}
\end{array} \right),$$
where $b,c,x$ $\in \mathbb{C}$ , $bc\neq0$ and $x\neq0$.

\item [(4)]  $\beta_4: M_2VB_n \rightarrow \text{GL}_n(\mathbb{C})$ such that
$$\beta_4(\sigma_i) =\left( \begin{array}{c|@{}c|c@{}}
   \begin{matrix}
     I_{i-1} 
   \end{matrix} 
      & \textbf{0} & \textbf{0} \\
      \hline
    \textbf{0} &\hspace{0.2cm} \begin{matrix}
\begin{matrix}
1 & 0\\
0 & 1
\end{matrix}
   		\end{matrix}  & \textbf{0}  \\
\hline
\textbf{0} & \textbf{0} & I_{n-i-1}
\end{array} \right), \
\beta_4(\rho_i^{0}) =\left( \begin{array}{c|@{}c|c@{}}
   \begin{matrix}
     I_{i-1} 
   \end{matrix} 
      & \textbf{0} & \textbf{0} \\
      \hline
    \textbf{0} &\hspace{0.2cm} \begin{matrix}
   		\begin{matrix}
0 & x \\
\frac{1}{x} & 0
\end{matrix}
   		\end{matrix}  & \textbf{0}  \\
\hline
\textbf{0} & \textbf{0} & I_{n-i-1}
\end{array} \right),$$
and
$$\beta_4(\rho_i^{1}) =\left( \begin{array}{c|@{}c|c@{}}
   \begin{matrix}
     I_{i-1} 
   \end{matrix} 
      & \textbf{0} & \textbf{0} \\
      \hline
    \textbf{0} &\hspace{0.2cm} \begin{matrix}
   		\begin{matrix}
1 & 0 \\
0 & 1
\end{matrix}
   		\end{matrix}  & \textbf{0}  \\
\hline
\textbf{0} & \textbf{0} & I_{n-i-1}
\end{array} \right),$$
where $x$ $\in \mathbb{C}$ and $x\neq0$.
   
\item[(5)] $\beta_5: M_2VB_n \rightarrow \text{GL}_n(\mathbb{C})$ such that
   $$\beta_5(\sigma_i) =\left( \begin{array}{c|@{}c|c@{}}
   \begin{matrix}
     I_{i-1} 
   \end{matrix} 
      & \textbf{0} & \textbf{0} \\
      \hline
    \textbf{0} &\hspace{0.2cm} \begin{matrix}
\begin{matrix}
0 & \frac{1-d}{c}\\
c & d
\end{matrix}

   		\end{matrix}  & \textbf{0}  \\
\hline
\textbf{0} & \textbf{0} & I_{n-i-1}
\end{array} \right),\
\beta_5(\rho_i^{0}) =\left( \begin{array}{c|@{}c|c@{}}
   \begin{matrix}
     I_{i-1} 
   \end{matrix} 
      & \textbf{0} & \textbf{0} \\
      \hline
    \textbf{0} &\hspace{0.2cm} \begin{matrix}
   		\begin{matrix}
0 & x \\
\frac{1}{x} & 0
\end{matrix}
   		\end{matrix}  & \textbf{0}  \\
\hline
\textbf{0} & \textbf{0} & I_{n-i-1}
\end{array} \right),$$
and
$$\beta_5(\rho_i^{1}) =\left( \begin{array}{c|@{}c|c@{}}
   \begin{matrix}
     I_{i-1} 
   \end{matrix} 
      & \textbf{0} & \textbf{0} \\
      \hline
    \textbf{0} &\hspace{0.2cm} \begin{matrix}
   		\begin{matrix}
1 & 0 \\
0 & 1
\end{matrix}
   		\end{matrix}  & \textbf{0}  \\
\hline
\textbf{0} & \textbf{0} & I_{n-i-1}
\end{array} \right),$$
where $c,d,x$ $\in \mathbb{C}$, $d\neq1$, $c\neq 0$ and $x\neq0$.

\item [(6)] $\beta_6: M_2VB_n \rightarrow \text{GL}_n(\mathbb{C})$ such that
 $$\beta_6(\sigma_i) =\left( \begin{array}{c|@{}c|c@{}}
   \begin{matrix}
     I_{i-1} 
   \end{matrix} 
      & \textbf{0} & \textbf{0} \\
      \hline
    \textbf{0} &\hspace{0.2cm} \begin{matrix}
\begin{matrix}
1-bc & b\\
c & 0
\end{matrix}

   		\end{matrix}  & \textbf{0}  \\
\hline
\textbf{0} & \textbf{0} & I_{n-i-1}
\end{array} \right),\
\beta_6(\rho_i^{0}) =\left( \begin{array}{c|@{}c|c@{}}
   \begin{matrix}
     I_{i-1} 
   \end{matrix} 
      & \textbf{0} & \textbf{0} \\
      \hline
    \textbf{0} &\hspace{0.2cm} \begin{matrix}
   		\begin{matrix}
0 & x \\
\frac{1}{x} & 0
\end{matrix}
   		\end{matrix}  & \textbf{0}  \\
\hline
\textbf{0} & \textbf{0} & I_{n-i-1}
\end{array} \right),$$
and
$$\beta_6(\rho_i^{1}) =\left( \begin{array}{c|@{}c|c@{}}
   \begin{matrix}
     I_{i-1} 
   \end{matrix} 
      & \textbf{0} & \textbf{0} \\
      \hline
    \textbf{0} &\hspace{0.2cm} \begin{matrix}
   		\begin{matrix}
0 & q \\
\frac{1}{q} & 0 
\end{matrix}
   		\end{matrix}  & \textbf{0}  \\
\hline
\textbf{0} & \textbf{0} & I_{n-i-1}
\end{array} \right),$$
where $b,c,x$ $\in \mathbb{C}$, $bc\neq0$, $x\neq0$ and $q\neq0$.

\item [(7)] $\beta_7: M_2VB_n \rightarrow \text{GL}_n(\mathbb{C})$ such that
$$\beta_7(\sigma_i) =\left( \begin{array}{c|@{}c|c@{}}
   \begin{matrix}
     I_{i-1} 
   \end{matrix} 
      & \textbf{0} & \textbf{0} \\
      \hline
    \textbf{0} &\hspace{0.2cm} \begin{matrix}
\begin{matrix}
0 & b\\
c & 0
\end{matrix}

   		\end{matrix}  & \textbf{0}  \\
\hline
\textbf{0} & \textbf{0} & I_{n-i-1}
\end{array} \right),\
\beta_7(\rho_i^{0}) =\left( \begin{array}{c|@{}c|c@{}}
   \begin{matrix}
     I_{i-1} 
   \end{matrix} 
      & \textbf{0} & \textbf{0} \\
      \hline
    \textbf{0} &\hspace{0.2cm} \begin{matrix}
   		\begin{matrix}
0 & x \\
\frac{1}{x} & 0
\end{matrix}
   		\end{matrix}  & \textbf{0}  \\
\hline
\textbf{0} & \textbf{0} & I_{n-i-1}
\end{array} \right),$$
and
$$\beta_7(\rho_i^{1}) =\left( \begin{array}{c|@{}c|c@{}}
   \begin{matrix}
     I_{i-1} 
   \end{matrix} 
      & \textbf{0} & \textbf{0} \\
      \hline
    \textbf{0} &\hspace{0.2cm} \begin{matrix}
   		\begin{matrix}
0 & q \\
\frac{1}{q} & 0
\end{matrix}
   		\end{matrix}  & \textbf{0}  \\
\hline
\textbf{0} & \textbf{0} & I_{n-i-1}
\end{array} \right),$$
where $b,c,x$ $\in \mathbb{C}$, $bc\neq0$, $x\neq0$ and $q\neq0$.

\item[(8)]  $\beta_8: M_2VB_n \rightarrow \text{GL}_n(\mathbb{C})$ such that
$$\beta_8(\sigma_i) =\left( \begin{array}{c|@{}c|c@{}}
   \begin{matrix}
     I_{i-1} 
   \end{matrix} 
      & \textbf{0} & \textbf{0} \\
      \hline
    \textbf{0} &\hspace{0.2cm} \begin{matrix}
\begin{matrix}
1 & 0\\
0 & 1
\end{matrix}

   		\end{matrix}  & \textbf{0}  \\
\hline
\textbf{0} & \textbf{0} & I_{n-i-1}
\end{array} \right),\
\beta_8(\rho_i^{0}) =\left( \begin{array}{c|@{}c|c@{}}
   \begin{matrix}
     I_{i-1} 
   \end{matrix} 
      & \textbf{0} & \textbf{0} \\
      \hline
    \textbf{0} &\hspace{0.2cm} \begin{matrix}
   		\begin{matrix}
0 & x \\
\frac{1}{x} & 0
\end{matrix}
   		\end{matrix}  & \textbf{0}  \\
\hline
\textbf{0} & \textbf{0} & I_{n-i-1}
\end{array} \right),$$
and
$$\beta_8(\rho_i^{1}) =\left( \begin{array}{c|@{}c|c@{}}
   \begin{matrix}
     I_{i-1} 
   \end{matrix} 
      & \textbf{0} & \textbf{0} \\
      \hline
    \textbf{0} &\hspace{0.2cm} \begin{matrix}
   		\begin{matrix}
0 & q \\
\frac{1}{q} & 0
\end{matrix}
   		\end{matrix}  & \textbf{0}  \\
\hline
\textbf{0} & \textbf{0} & I_{n-i-1}
\end{array} \right),$$
where $x$ $\in \mathbb{C}$, $x\neq0$ and $q\neq0$.
   
\item[(9)] $\beta_9: M_2VB_n \rightarrow \text{GL}_n(\mathbb{C})$ such that
 $$\beta_9(\sigma_i) =\left( \begin{array}{c|@{}c|c@{}}
   \begin{matrix}
     I_{i-1} 
   \end{matrix} 
      & \textbf{0} & \textbf{0} \\
      \hline
    \textbf{0} &\hspace{0.2cm} \begin{matrix}
\begin{matrix}
0 & \frac{1-d}{c}\\
c & d
\end{matrix}

   		\end{matrix}  & \textbf{0}  \\
\hline
\textbf{0} & \textbf{0} & I_{n-i-1}
\end{array} \right),\
\beta_9(\rho_i^{0}) =\left( \begin{array}{c|@{}c|c@{}}
   \begin{matrix}
     I_{i-1} 
   \end{matrix} 
      & \textbf{0} & \textbf{0} \\
      \hline
    \textbf{0} &\hspace{0.2cm} \begin{matrix}
   		\begin{matrix}
0 & x \\
\frac{1}{x} & 0
\end{matrix}
   		\end{matrix}  & \textbf{0}  \\
\hline
\textbf{0} & \textbf{0} & I_{n-i-1}
\end{array} \right),$$
and
$$\beta_9(\rho_i^{1}) =\left( \begin{array}{c|@{}c|c@{}}
   \begin{matrix}
     I_{i-1} 
   \end{matrix} 
      & \textbf{0} & \textbf{0} \\
      \hline
    \textbf{0} &\hspace{0.2cm} \begin{matrix}
   		\begin{matrix}
0 & q \\
\frac{1}{q} & 0
\end{matrix}
   		\end{matrix}  & \textbf{0}  \\
\hline
\textbf{0} & \textbf{0} & I_{n-i-1}
\end{array} \right),$$
where $c,d,x$ $\in \mathbb{C}$, $d\neq1$, $c\neq 0$, $x\neq0$ and $q\neq0$.
\end{itemize}

Now we work on $M_3VB_n$ which has four family of generators $\sigma_i$'s, $\rho_i^{0}$'s, $\rho_i^{1}$'s, and $ \rho_i^{2}$'s., Here also we only need to consider the following relations of generators of $M_3VB_n$, and rest other relations imply similar equations.
     \begin{align}
         (\rho_1^{0})^2 & = e,  \\
         (\rho_1^{1})^2 & = e,\\
         (\rho_1^{2})^2 & = e,\\
         \sigma_1 \sigma_2 \sigma_1 & = \sigma_2 \sigma_1 \sigma_2,\\
         \rho_1^{0} \rho_2^{0} \rho_1^{0} & = \rho_2^{0} \rho_1^{0} \rho_2^{0},\\
         \rho_1^{1} \rho_2^{1} \rho_1^{1} & = \rho_2^{1} \rho_1^{1} \rho_2^{1},\\
         \rho_1^{2} \rho_2^{2} \rho_1^{2} & = \rho_2^{2} \rho_1^{2} \rho_2^{2},\\
         \sigma_1 \rho_2^{0} \rho_1^{0} & = \rho_2^{0} \rho_1^{0} \sigma_2,\\
         \rho_1^{0} \rho_2^{0} \rho_1^{1} & = \rho_2^{1} \rho_1^{0} \rho_2^{0},\\
         \rho_1^{0} \rho_2^{0} \rho_1^{2} & = \rho_2^{2} \rho_1^{0} \rho_2^{0}.
    \end{align}

Similarly, to determine all complex homogeneous $2$-local representations of $M_3VB_n$, we substitute the generator images under the representation $\beta$ into the above relations of the group. This substitution produces a system of 61 algebraic equations involving 16 unknowns.
Solving this system of equations yields seventeen different solutions, each of which corresponds to a representation of the group  $M_3VB_n$.

Similarly, for the general case of $M_kVB_n$ by solving the analogous set of equations we obtain all  $2^{k+1}+1$ representation of $M_kVB_n$.
\end{proof}

Now we study the faithfulness and the irreducibility of the above representations in some cases.

\begin{remark}\label{rem:1}
Note that among all the \( 2^{k+1} + 1 \) representations, only the following three are non-trivial in the sense that no entire family of generators is mapped directly to the identity element of \( \mathrm{GL}_n(\mathbb{C}) \).
 $$\beta(\sigma_i) =\left( \begin{array}{c|@{}c|c@{}}
   \begin{matrix}
     I_{i-1} 
   \end{matrix} 
      & \textbf{0} & \textbf{0} \\
      \hline
    \textbf{0} &\hspace{0.2cm} \begin{matrix}
   		S
   		\end{matrix}  & \textbf{0}  \\
\hline
\textbf{0} & \textbf{0} & I_{n-i-1}
\end{array} \right), \text{ and \  } 
\beta(\rho_i^{\alpha}) =\left( \begin{array}{c|@{}c|c@{}}
   \begin{matrix}
     I_{i-1} 
   \end{matrix} 
      & \textbf{0} & \textbf{0} \\
      \hline
    \textbf{0} &\hspace{0.2cm} \begin{matrix}
        \begin{matrix} 0 & x_{\alpha} \\ \frac{1}{x_{\alpha}} & 0 \end{matrix}
   		\end{matrix}  & \textbf{0}  \\
\hline
\textbf{0} & \textbf{0} & I_{n-i-1}
\end{array} \right),$$
for $i=1,2,\dots,n-1$ and $ 0 \leq \alpha \leq k-1$,\\
where\[
S \in \left\{ 
\begin{pmatrix} 1 - bc & b \\ c & 0 \end{pmatrix},\ 
\begin{pmatrix} 0 & b \\ c & 0 \end{pmatrix},\ 
\begin{pmatrix} 0 & -\frac{d-1}{c} \\ c & d \end{pmatrix} 
\right\}\  
\text{with } bc \neq 0,\ d \neq 1,
\] 
and $x_{\alpha} \neq 0.$
For example, in case of $M_2VB_n$ the representation \[ \beta_7: M_2VB_n \rightarrow \text{GL}_n(\mathbb{C}) \] defined on generators as
$$\beta_7(\sigma_i) =\left( \begin{array}{c|@{}c|c@{}}
   \begin{matrix}
     I_{i-1} 
   \end{matrix} 
      & \textbf{0} & \textbf{0} \\
      \hline
    \textbf{0} &\hspace{0.2cm} \begin{matrix}
\begin{matrix}
0 & b\\
c & 0
\end{matrix}

   		\end{matrix}  & \textbf{0}  \\
\hline
\textbf{0} & \textbf{0} & I_{n-i-1}
\end{array} \right), \
\beta_7(\rho_i^{0}) =\left( \begin{array}{c|@{}c|c@{}}
   \begin{matrix}
     I_{i-1} 
   \end{matrix} 
      & \textbf{0} & \textbf{0} \\
      \hline
    \textbf{0} &\hspace{0.2cm} \begin{matrix}
   		\begin{matrix}
0 & x \\
\frac{1}{x} & 0
\end{matrix}
   		\end{matrix}  & \textbf{0}  \\
\hline
\textbf{0} & \textbf{0} & I_{n-i-1}
\end{array} \right),$$
and
$$\beta_7(\rho_i^{1}) =\left( \begin{array}{c|@{}c|c@{}}
   \begin{matrix}
     I_{i-1} 
   \end{matrix} 
      & \textbf{0} & \textbf{0} \\
      \hline
    \textbf{0} &\hspace{0.2cm} \begin{matrix}
   		\begin{matrix}
0 & q \\
\frac{1}{q} & 0
\end{matrix}
   		\end{matrix}  & \textbf{0}  \\
\hline
\textbf{0} & \textbf{0} & I_{n-i-1}
\end{array} \right),$$
where $b,c,x,q \in \mathbb{C}$, $bc\neq0$, $x\neq0$ and $q\neq0$  is a representation such that the generators $\sigma_i, \rho_i^0 \text{ and } \rho_i^1$ are not directly mapped to identity.
 \end{remark}

 \begin{theorem}
     All complex homogeneous representations of $M_kVB_n$, except for the three representations mentioned in Remark~\ref{rem:1} (which may or may not be faithful), are unfaithful.
\end{theorem}
\begin{proof}
    Let $\beta$ be any representation among the $2^{k+1} +1$ representation under consideration, excluding the three specified in the preceding remark. By the definition of these representation, for each such $\beta$ , there exists at least one nontrivial family of generators such as $\sigma_i$, $\rho_i^0$, or $\rho_i^1$, that  is mapped  entirely to the identity matrix in  $\text{GL}_n(\mathbb{C})$.
\end{proof}

The faithfulness of the representations mentioned in Remark~\ref{rem:1} is not guaranteed; however, they are known to be unfaithful under certain conditions, one of which is established in Theorem~3.4 below.

\begin{theorem}
     The three representations of $M_kVB_n$ highlighted in remark 3.2 are unfaithful under the following condition,
     \begin{itemize}
         \item[(i)]  When S= $\begin{pmatrix} 0 & b \\ c & 0 \end{pmatrix}$, the representation is unfaithful if, for any fixed $\alpha$, we have $x_{\alpha}= b \text{ and } \frac{1}{x_{\alpha}}= c$ .
         \item[(ii)]  When S=$\begin{pmatrix} 1 - bc & b \\ c & 0 \end{pmatrix} \text{ or }\  
\begin{pmatrix} 0 & -\frac{d-1}{c} \\ c & d \end{pmatrix}$, the representation is unfaithful if the values  $x_{\alpha}'s$ are equal for any two distinct indices $\alpha'$s.
     \end{itemize}
\end{theorem}
\begin{proof}
    For case $(i)$ the family $\sigma_i$'s and $\rho_i^{\alpha}$'s, for $i=1,2,\dots,n-1$ and any fixed $\alpha$, say $\alpha=0$ are mapped to same matrix, i.e., $$\beta(\sigma_i) =\left( \begin{array}{c|@{}c|c@{}}
   \begin{matrix}
     I_{i-1} 
   \end{matrix} 
      & \textbf{0} & \textbf{0} \\
      \hline
    \textbf{0} &\hspace{0.2cm} \begin{matrix}
   		\begin{matrix} 0 & b \\ c & 0 \end{matrix}
   		\end{matrix}  & \textbf{0}  \\
\hline
\textbf{0} & \textbf{0} & I_{n-i-1}
\end{array} \right), \text{ and } 
\beta(\rho_i^{0}) =\left( \begin{array}{c|@{}c|c@{}}
   \begin{matrix}
     I_{i-1} 
   \end{matrix} 
      & \textbf{0} & \textbf{0} \\
      \hline
    \textbf{0} &\hspace{0.2cm} \begin{matrix}
        \begin{matrix} 0 & b \\ c & 0 \end{matrix}
   		\end{matrix}  & \textbf{0}  \\
\hline
\textbf{0} & \textbf{0} & I_{n-i-1}
\end{array} \right),$$
for $i=1,2,\dots,n-1$. Hence, they are not faithful.\\ The proof of case $(ii)$  is similar to case $(i)$. 
    \end{proof}

In the following theorem, we examine the irreducibility of the representations of $M_kVB_n$ introduced in Theorem~\ref{th3.1}. Before proceeding, we present a lemma that will play a key role in the subsequent proof.
\begin{lemma} \label{llem}
  Consider $n \geq 2$ and let \( \theta : G \to \mathrm{GL}_n(V) \) be a representation of a group $G$ on a vector space $V$. Let $H$ be any subgroup of $G$. If \(\theta|_H : H \to \mathrm{GL}_n(V) \) is irreducible, then \( \theta : G \to \mathrm{GL}_n(V) \) is also irreducible.
\end{lemma}
\begin{proof}
     Suppose that \( \theta : G \to \mathrm{GL}_n(V) \) is reducible. Then, there exists an invariant subspace \( S \subseteq V \) such that \( \theta(g)s \in S \) for every \( s \in S \) and \( g \in G \). This implies that \( \theta(h)s \in S \) for every \( s \in S \) and \( h \in H \subseteq G \). Hence, \( S \) is an invariant subspace of \( V \) with respect to \( \theta|_H : H \to \mathrm{GL}_n(V) \). This implies that $\theta|_H$ is reducible, which is a contradiction.
\end{proof}

\begin{theorem}
Consider the $2^{k+1}+1$ representations of $M_kVB_n$ given in Theorem \ref{th3.1}. These representations have the form:
$$\beta(\sigma_i) =\left( \begin{array}{c|@{}c|c@{}}
   \begin{matrix}
     I_{i-1} 
   \end{matrix} 
      & \textbf{0} & \textbf{0} \\
      \hline
    \textbf{0} &\hspace{0.2cm} \begin{matrix}
   		S
   		\end{matrix}  & \textbf{0}  \\
\hline
\textbf{0} & \textbf{0} & I_{n-i-1}
\end{array} \right), \text{ and } 
\beta(\rho_i^{\alpha}) =\left( \begin{array}{c|@{}c|c@{}}
   \begin{matrix}
     I_{i-1} 
   \end{matrix} 
      & \textbf{0} & \textbf{0} \\
      \hline
    \textbf{0} &\hspace{0.2cm} \begin{matrix}
   		X_{\alpha}
   		\end{matrix}  & \textbf{0}  \\
\hline
\textbf{0} & \textbf{0} & I_{n-i-1}
\end{array} \right),$$
where \( S \) and \( X_{\alpha} \) are chosen from the following families of matrices:  
\[
S \in \left\{ 
\begin{pmatrix} 1 & 0 \\ 0 & 1 \end{pmatrix},\ 
\begin{pmatrix} 1 - bc & b \\ c & 0 \end{pmatrix},\ 
\begin{pmatrix} 0 & b \\ c & 0 \end{pmatrix},\ 
\begin{pmatrix} 0 & -\frac{d-1}{c} \\ c & d \end{pmatrix} 
\right\}\ \text{with } bc \neq 0,\ d \neq 1,
\]
\[
X_{\alpha} \in \left\{
\begin{pmatrix} 0 & x_{\alpha} \\ \frac{1}{x_{\alpha}} & 0 \end{pmatrix},\ 
\begin{pmatrix} 1 & 0 \\ 0 & 1 \end{pmatrix}
\right\},\ \text{for } 1 \leq \alpha \leq k-1\ \text{with } x_{\alpha} \neq 0. \]
 and
\[
 X_0=\begin{pmatrix} 0 & x_{0} \\ \frac{1}{x_{0}} & 0 \end{pmatrix} \text{ with } x_0\neq0.\]
Then, the following hold true.
\begin{itemize}
    \item[(a)] If $S=\begin{pmatrix} 1 & 0 \\ 0 & 1 \end{pmatrix},$ we consider the following two sub cases:
    \begin{itemize}
         \item[(i)] When $X_\alpha=\begin{pmatrix} 0 & x_{\alpha} \\ \frac{1}{x_{\alpha}} & 0 \end{pmatrix},$ then $\beta$ is reducible if $x_\alpha=1$ for all $0\leq \alpha \leq k-1$.
         \item[(ii)] When $X_\alpha=\begin{pmatrix} 1 & 0 \\ 0 & 1 \end{pmatrix}$, then $\beta$ is reducible if $x_0=1$.
    \end{itemize}
   \item[(b)] If $S=\begin{pmatrix} 1 - bc & b \\ c & 0 \end{pmatrix},$ we consider the following two sub cases:
    \begin{itemize}
        \item[(i)] When $X_\alpha=\begin{pmatrix} 0 & x_{\alpha} \\ \frac{1}{x_{\alpha}} & 0 \end{pmatrix},$ then $\beta$ is reducible if $cx_\alpha=1$ for all $0\leq \alpha \leq k-1$ and $n\geq 6$.
    \item[(ii)] When $X_\alpha=\begin{pmatrix} 1 & 0 \\ 0 & 1 \end{pmatrix}$, then $\beta$ is reducible if $cx_0=1$ and $n\geq 6$.
    \end{itemize}
    \item[(c)] If $S=\begin{pmatrix} 0 & b \\ c & 0 \end{pmatrix}$, we consider the following two sub cases:
    \begin{itemize}
        \item[(i)] When $X_\alpha=\begin{pmatrix} 0 & x_{\alpha} \\ \frac{1}{x_{\alpha}} & 0 \end{pmatrix}$, then:
    \begin{itemize}
        \item If $bc\neq 1$, then $\beta$ is irreducible.
        \item If $b=c=x_\alpha=1$ for all $0\leq \alpha \leq k-1$, then $\beta$ is reducible.
    \end{itemize}
    \item[(ii)] When $X_\alpha=\begin{pmatrix} 1 & 0 \\ 0 & 1 \end{pmatrix}$, then:
     \begin{itemize}
        \item If $bc\neq 1$, then $\beta$ is irreducible.
        \item If $b=c=x_0=1$, then $\beta$ is reducible.
        \end{itemize}
    \end{itemize}
    \item[(d)] If $S=\begin{pmatrix} 0 & -\frac{d-1}{c} \\ c & d \end{pmatrix}$, we consider the following two sub cases:
    \begin{itemize}
        \item[(i)] When $X_\alpha=\begin{pmatrix} 0 & x_{\alpha} \\ \frac{1}{x_{\alpha}} & 0 \end{pmatrix}$, then $\beta$ is reducible if $cx_\alpha=1$ for all $0\leq \alpha \leq k-1$ and $n\geq 6$.
        \item[(ii)] When $X=\begin{pmatrix} 1 & 0 \\ 0 & 1 \end{pmatrix}$, then $\beta$ is reducible if $cx_0=1$ and $n\geq 6$.
    \end{itemize}
\end{itemize}
\end{theorem}
\begin{proof}
We consider each case separately.
\begin{itemize}
    \item[(a)]
    \begin{itemize}
        \item[(i)] In this case, we can see that the column vector $(1,1,\ldots,1)^T$ is invariant under $\beta(\sigma_i)$ and $\beta(\rho_i^\alpha)$ for all $1\leq i\leq n-1$ and $0\leq k \leq \alpha-1$. Thus, $\beta$ is reducible.
        \item[(ii)] The proof is similar to (a) first part.
    \end{itemize}
    \item[(b)] 
    \begin{itemize}
        \item[(i)] In this case, we consider an equivalent representation to $\beta.$ Consider the diagonal matrix $T=\text{diag}(c^{1-n},c^{2-n},\ldots, c,1)$, where $\text{diag}(t_1,t_2,\dots,t_n)$ is a diagonal $n\times n$ matrix with $t_{ii}=t_i$. Consider the equivalent representation $\beta'$ of $\beta$ given by: $\beta'(\sigma_i)=T^{-1}\beta(\sigma_i)T$ and $\beta'(\rho_i^\alpha)=T^{-1}\beta(\rho_i^\alpha)T$ for all $0\leq \alpha \leq k-1$ and $1\leq i \leq n-1$. Direct computations implies that 
$$\beta'(\sigma_i) =\left( \begin{array}{c|@{}c|c@{}}
   \begin{matrix}
     I_{i-1} 
   \end{matrix} 
      & 0 & 0 \\
      \hline
    0 &\hspace{0.2cm} \begin{matrix}
   		1-bc & bc\\
   		1 & 0
   		\end{matrix}  & 0  \\
\hline
0 & 0 & I_{n-i-1}
\end{array} \right), \text{ and } \beta'(\rho_i^\alpha) =\left( \begin{array}{c|@{}c|c@{}}
   \begin{matrix}
     I_{i-1} 
   \end{matrix} 
      & 0 & 0 \\
      \hline
    0 &\hspace{0.2cm} \begin{matrix}
   		 0 & cx_\alpha\\
   		\frac{1}{cx_\alpha} & 0
   		\end{matrix}  & 0  \\
\hline
0 & 0 & I_{n-i-1}
\end{array} \right),$$
$\text{for } 0\leq \alpha \leq k-1 \text{ and } 1\leq i\leq n-1$. Now, since $cx_\alpha=1$ for all $0\leq \alpha \leq k-1$, we can easily see that the column vector $(1,1,\ldots,1)^T$ is invariant under $\beta'(\sigma_i)$ and $\beta'(\rho_i^\alpha)$ for all $0\leq k \leq \alpha-1$ and $1\leq i\leq n-1$. Therefore, $\beta'$ is reducible, and so $\beta$ is reducible.
 \item[(ii)]  The proof is similar to (b) first part.
    \end{itemize}
 \item[(c)]
\begin{itemize}
    \item[(i)] In this case, we consider two sub cases:
    \begin{itemize}
        \item If $bc\neq 1$, then the restriction of $\beta$ to $B_n$ is irreducible by Theorem \ref{ThCh2}. Hence, $\beta$ is irreducible by Lemma \ref{llem}.
        \item If if $b=c=x_\alpha=1$ for all $0\leq \alpha \leq k-1$, then  we can see that the column vector $(1,1,\ldots,1)^T$ is invariant under $\beta(\sigma_i)$ and $\beta(\rho_i^\alpha)$ for all $0\leq k \leq \alpha-1$ and $1\leq i\leq n-1$, and so $\beta$ is reducible in this sub case.
   \end{itemize}
    \item[(ii)] The proof is similar to (c) first part.
    \end{itemize}
    \item[(d)] The proof is similar to (b) part.
    \end{itemize}
\end{proof}

\section{Complex Homogeneous Local Representations of $M_kWB_n$}

In this section, we aim to determine all the complex homogeneous $2$-local representations of the multi-welded braid group $M_kWB_n$ for $n\geq 3$ and $k > 1$.

\begin{theorem}
    Let $\zeta: M_kWB_n \rightarrow \text{GL}_{n}(\mathbb{C})$ be a complex homogeneous $2$-local representation of $M_kWB_n$. Then, $\zeta$ is equivalent to one of the following $3\cdot2^{k-1}+1$ representations. Among these, one is trivial representation in which $\zeta$ sends each generator to the identity matrix, and the remaining $ 3 \cdot 2^{k-1}$ representations are of the form:
     $$\zeta(\sigma_i) =\left( \begin{array}{c|@{}c|c@{}}
   \begin{matrix}
     I_{i-1} 
   \end{matrix} 
      & \textbf{0} & \textbf{0} \\
      \hline
    \textbf{0} &\hspace{0.2cm} \begin{matrix}
   		A
   		\end{matrix}  & \textbf{0}  \\
\hline
\textbf{0} & \textbf{0} & I_{n-i-1}
\end{array} \right), \ 
 \text{ and } \
\zeta(\rho_i^{\alpha}) =\left( \begin{array}{c|@{}c|c@{}}
   \begin{matrix}
     I_{i-1} 
   \end{matrix} 
      & \textbf{0} & \textbf{0} \\
      \hline
    \textbf{0} &\hspace{0.2cm} \begin{matrix}
   		Y_{\alpha}
   		\end{matrix}  & \textbf{0}  \\
\hline
\textbf{0} & \textbf{0} & I_{n-i-1}
\end{array} \right),$$

for $i=1,2, \dots, n-1$, where \( A \) and \( Y_{\alpha} \) are chosen from the following families of matrices:  
\[
A \in \left\{ 
\begin{pmatrix} 1 - bc & b \\ c & 0 \end{pmatrix},\ 
\begin{pmatrix} 0 & b \\ c & 0 \end{pmatrix},\ 
\begin{pmatrix}
0 & b \\
\frac{1 - d}{b} & d
\end{pmatrix} 
\right\} \ \ \text{with } bc \neq 0,\ d \neq 1,
\]
and
\[
Y_{\alpha} \in \left\{
\begin{pmatrix} 0 & x_{\alpha} \\ \frac{1}{x_{\alpha}} & 0 \end{pmatrix},\ 
\begin{pmatrix}
0 & b \\
\frac{1}{b} & 0
\end{pmatrix},\
\begin{pmatrix} 1 & 0 \\ 0 & 1 \end{pmatrix}
\right\} \ \ \text{for } 0 \leq \alpha \leq k-1,\ \text{with } x_{\alpha} \neq 0, ~ b \neq 0, \] and satisfy the conditions:
\begin{itemize}
    \item[(i)] If $A$= $\begin{pmatrix} 1 - bc & b \\ c & 0 \end{pmatrix}$\text{ or}\ 
$A=\begin{pmatrix}
0 & b \\
\frac{1 - d}{b} & d
\end{pmatrix}$,\  then the matrices $Y_{\alpha}$ is given by:
\begin{itemize}
    \item[(a)] $Y_{0}= \begin{pmatrix}
  0 & b \\
  \frac{1}{b} & 0
  \end{pmatrix}$, and 
  \item[(b)] $Y_{\alpha}$=$\begin{pmatrix}
  0 & b \\
  \frac{1}{b} & 0
  \end{pmatrix}$ or\ $Y_{\alpha}=\begin{pmatrix} 1 & 0 \\ 0 & 1 \end{pmatrix}$ for $1 \leq \alpha \leq k-1$.
\end{itemize}
  \item[(ii)] if $A$= $\begin{pmatrix} 0 & b \\ c & 0 \end{pmatrix}$, then the matrices $Y_{\alpha}$ is given by:
\begin{itemize}
\item[(a)]  $Y_0 = \begin{pmatrix} 0 & x_0 \\ \frac{1}{x_0} & 0 \end{pmatrix}$, and 
    \item[(b)] $Y_{\alpha}=\begin{pmatrix} 0 & x_{\alpha} \\ \frac{1}{x_{\alpha}} & 0 \end{pmatrix}$ or $Y_{\alpha}=\begin{pmatrix} 1 & 0 \\ 0 & 1 \end{pmatrix}$ for $1 \leq \alpha \leq k-1$.
    
\end{itemize}
\end{itemize}
\end{theorem}

\begin{proof}
    Similar to Theorem 3.1, here we have some extra relations because of new relation F1 (over-forbidden move) which are: \[\sigma_i \sigma_{i+1} \rho_i^{\alpha} = \rho_{i+1}^{\alpha} \sigma_i \sigma_{i+1},  \hspace{0.875cm} i=1,2,\ldots ,n-2, \hspace{0.1cm} 0 \leq \alpha \leq k-1.\]
    In the case of \( M_2WB_n \), only the two additional relations (4.1)–(4.2) are added to the defining relations of \( M_2VB_n \) given in (3.1)–(3.8), resulting in fourteen new equations, namely (4.3)–(4.16).
\begin{align}
  \sigma_1 \sigma_2 \rho_1^{0} & =  \rho_2^{0} \sigma_1 \sigma_2\\
    \sigma_1 \sigma_2 \rho_1^{1} & = \rho_2^{1} \sigma_1 \sigma_2 
\end{align}
\begin{tcolorbox}[
  title=System of Equations Corresponding to Relations (4.1)–(4.2),
  colframe=black,
  colback=white,
  sharp corners,
  fonttitle=\bfseries,
]

\large

\begin{parcolumns}[colwidths={=0.48\textwidth}, rulebetween]{2}

\colchunk{
\textbf{\boldmath (i) $\sigma_1 \sigma_2 \rho_1^{0} =  \rho_2^{0} \sigma_1 \sigma_2$} \\
\rule{\linewidth}{0.4pt}

\begin{align}
aw + aby &= a, \\
ax + abz &= ab, \\
cw + ady &= cw, \\
cx + adz &= cx + adw, \\
dx + bdw &= bd, \\
cz + ady &= cz, \\
dz + bdy &= d,
\end{align}
}

\colchunk{
\textbf{\boldmath (ii) $\sigma_1 \sigma_2 \rho_1^{1} = \rho_2^{1} \sigma_1 \sigma_2$} \\
\rule{\linewidth}{0.4pt}

\begin{align}
ap + abr &= a, \\
aq + abs &= ab, \\
cp + adr &= cp, \\
cq + ads &= cq + adp, \\
dq + bdp &= bd, \\
cs + adr &= cs, \\
ds + bdr &= d.
\end{align}
}

\end{parcolumns}
\end{tcolorbox}

Solving equations (3.8)-(3.50) and (4.3)-(4.16), we get seven solutions which gives seven representations of $M_2WB_n$ which are:
\begin{itemize}
    \item [(1)] $\zeta_1: M_2WB_n \rightarrow \text{GL}_n(\mathbb{C})$ such that
    $$\zeta_1(\sigma_i) =\left( \begin{array}{c|@{}c|c@{}}
   \begin{matrix}
     I_{i-1} 
   \end{matrix} 
      & \textbf{0} & \textbf{0} \\
      \hline
    \textbf{0} &\hspace{0.2cm} \begin{matrix}
\begin{matrix}
1 & 0 \\
0 & 1
\end{matrix}

   		\end{matrix}  & \textbf{0}  \\
\hline
\textbf{0} & \textbf{0} & I_{n-i-1}
\end{array} \right), \
\zeta_1(\rho_i^{0}) =\left( \begin{array}{c|@{}c|c@{}}
   \begin{matrix}
     I_{i-1} 
   \end{matrix} 
      & \textbf{0} & \textbf{0} \\
      \hline
    \textbf{0} &\hspace{0.2cm} \begin{matrix}
   		\begin{matrix}
1 & 0 \\
0 & 1
\end{matrix}
   		\end{matrix}  & \textbf{0}  \\
\hline
\textbf{0} & \textbf{0} & I_{n-i-1}
\end{array} \right),$$
and
$$\zeta_1(\rho_i^{1}) =\left( \begin{array}{c|@{}c|c@{}}
   \begin{matrix}
     I_{i-1} 
   \end{matrix} 
      & \textbf{0} & \textbf{0} \\
      \hline
    \textbf{0} &\hspace{0.2cm} \begin{matrix}
   		\begin{matrix}
1 & 0 \\
0 & 1
\end{matrix}
   		\end{matrix}  & \textbf{0}  \\
\hline
\textbf{0} & \textbf{0} & I_{n-i-1}
\end{array} \right),$$ 

\item [(2)] $\zeta_2: M_2WB_n \rightarrow \text{GL}_n(\mathbb{C})$ such that
    $$\zeta_2(\sigma_i) =\left( \begin{array}{c|@{}c|c@{}}
   \begin{matrix}
     I_{i-1} 
   \end{matrix} 
      & \textbf{0} & \textbf{0} \\
      \hline
    \textbf{0} &\hspace{0.2cm} \begin{matrix}
\begin{matrix}
1-bc & b\\
c & 0
\end{matrix}

   		\end{matrix}  & \textbf{0}  \\
\hline
\textbf{0} & \textbf{0} & I_{n-i-1}
\end{array} \right), \
\zeta_2(\rho_i^{0}) =\left( \begin{array}{c|@{}c|c@{}}
   \begin{matrix}
     I_{i-1} 
   \end{matrix} 
      & \textbf{0} & \textbf{0} \\
      \hline
    \textbf{0} &\hspace{0.2cm} \begin{matrix}
   		\begin{matrix}
0 & b \\
\frac{1}{b} & 0
\end{matrix}
   		\end{matrix}  & \textbf{0}  \\
\hline
\textbf{0} & \textbf{0} & I_{n-i-1}
\end{array} \right),$$
and
$$\zeta_2(\rho_i^{1}) =\left( \begin{array}{c|@{}c|c@{}}
   \begin{matrix}
     I_{i-1} 
   \end{matrix} 
      & \textbf{0} & \textbf{0} \\
      \hline
    \textbf{0} &\hspace{0.2cm} \begin{matrix}
   		\begin{matrix}
1 & 0 \\
0 & 1
\end{matrix}
   		\end{matrix}  & \textbf{0}  \\
\hline
\textbf{0} & \textbf{0} & I_{n-i-1}
\end{array} \right),$$
where $b,c \in \mathbb{C}$ and $bc\neq0$.
   
\item [(3)] $\zeta_3: M_2WB_n \rightarrow \text{GL}_n(\mathbb{C})$ such that
   $$\zeta_3(\sigma_i) =\left( \begin{array}{c|@{}c|c@{}}
   \begin{matrix}
     I_{i-1} 
   \end{matrix} 
      & \textbf{0} & \textbf{0} \\
      \hline
    \textbf{0} &\hspace{0.2cm} \begin{matrix}
\begin{matrix}
0 & b\\
c & 0
\end{matrix}

   		\end{matrix}  & \textbf{0}  \\
\hline
\textbf{0} & \textbf{0} & I_{n-i-1}
\end{array} \right), \
\zeta_3(\rho_i^{0}) =\left( \begin{array}{c|@{}c|c@{}}
   \begin{matrix}
     I_{i-1} 
   \end{matrix} 
      & \textbf{0} & \textbf{0} \\
      \hline
    \textbf{0} &\hspace{0.2cm} \begin{matrix}
   		\begin{matrix}
0 & x \\
\frac{1}{x} & 0
\end{matrix}
   		\end{matrix}  & \textbf{0}  \\
\hline
\textbf{0} & \textbf{0} & I_{n-i-1}
\end{array} \right),$$
and
$$\zeta_3(\rho_i^{1}) =\left( \begin{array}{c|@{}c|c@{}}
   \begin{matrix}
     I_{i-1} 
   \end{matrix} 
      & \textbf{0} & \textbf{0} \\
      \hline
    \textbf{0} &\hspace{0.2cm} \begin{matrix}
   		\begin{matrix}
1 & 0 \\
0 & 1
\end{matrix}
   		\end{matrix}  & \textbf{0}  \\
\hline
\textbf{0} & \textbf{0} & I_{n-i-1}
\end{array} \right),$$
where $b,c,x \in \mathbb{C}$ , $bc\neq0$ and $x\neq0$.
   
\item[(4)] $\zeta_4: M_2WB_n \rightarrow \text{GL}_n(\mathbb{C})$ such that
   $$\zeta_4(\sigma_i) =\left( \begin{array}{c|@{}c|c@{}}
   \begin{matrix}
     I_{i-1} 
   \end{matrix} 
      & \textbf{0} & \textbf{0} \\
      \hline
    \textbf{0} &\hspace{0.2cm} \begin{matrix}
\begin{matrix}
0 & b\\
\frac{1-d}{b} & d
\end{matrix}

	\end{matrix}  & \textbf{0}  \\
\hline
\textbf{0} & \textbf{0} & I_{n-i-1}
\end{array} \right),\ 
\zeta_4(\rho_i^{0}) =\left( \begin{array}{c|@{}c|c@{}}
   \begin{matrix}
     I_{i-1} 
   \end{matrix} 
      & \textbf{0} & \textbf{0} \\
      \hline
    \textbf{0} &\hspace{0.2cm} \begin{matrix}
   		\begin{matrix}
0 & b \\
\frac{1}{b} & 0
\end{matrix}
   		\end{matrix}  & \textbf{0}  \\
\hline
\textbf{0} & \textbf{0} & I_{n-i-1}
\end{array} \right),$$
and
$$\zeta_4(\rho_i^{1}) =\left( \begin{array}{c|@{}c|c@{}}
   \begin{matrix}
     I_{i-1} 
   \end{matrix} 
      & \textbf{0} & \textbf{0} \\
      \hline
    \textbf{0} &\hspace{0.2cm} \begin{matrix}
   		\begin{matrix}
1 & 0 \\
0 & 1
\end{matrix}
   		\end{matrix}  & \textbf{0}  \\
\hline
\textbf{0} & \textbf{0} & I_{n-i-1}
\end{array} \right),$$
where $b,d \in \mathbb{C}$, $d\neq1$ and $b\neq0$.

\item [(5)] $\zeta_5: M_2WB_n \rightarrow \text{GL}_n(\mathbb{C})$ such that
 $$\zeta_5(\sigma_i) =\left( \begin{array}{c|@{}c|c@{}}
   \begin{matrix}
     I_{i-1} 
   \end{matrix} 
      & \textbf{0} & \textbf{0} \\
      \hline
    \textbf{0} &\hspace{0.2cm} \begin{matrix}
\begin{matrix}
1-bc & b\\
c & 0
\end{matrix}

   		\end{matrix}  & \textbf{0}  \\
\hline
\textbf{0} & \textbf{0} & I_{n-i-1}
\end{array} \right),\
\zeta_5(\rho_i^{0}) =\left( \begin{array}{c|@{}c|c@{}}
   \begin{matrix}
     I_{i-1} 
   \end{matrix} 
      & \textbf{0} & \textbf{0} \\
      \hline
    \textbf{0} &\hspace{0.2cm} \begin{matrix}
   		\begin{matrix}
0 & b \\
\frac{1}{b} & 0
\end{matrix}
   		\end{matrix}  & \textbf{0}  \\
\hline
\textbf{0} & \textbf{0} & I_{n-i-1}
\end{array} \right),$$
and
$$\zeta_5(\rho_i^{1}) =\left( \begin{array}{c|@{}c|c@{}}
   \begin{matrix}
     I_{i-1} 
   \end{matrix} 
      & \textbf{0} & \textbf{0} \\
      \hline
    \textbf{0} &\hspace{0.2cm} \begin{matrix}
   		\begin{matrix}
0 & b \\
\frac{1}{b} & 0 
\end{matrix}
   		\end{matrix}  & \textbf{0}  \\
\hline
\textbf{0} & \textbf{0} & I_{n-i-1}
\end{array} \right),$$
where $b,c \in \mathbb{C}$ and $bc\neq0$.

\item [(6)] $\zeta_6: M_2WB_n \rightarrow \text{GL}_n(\mathbb{C})$ such that
$$\zeta_6(\sigma_i) =\left( \begin{array}{c|@{}c|c@{}}
   \begin{matrix}
     I_{i-1} 
   \end{matrix} 
      & \textbf{0} & \textbf{0} \\
      \hline
    \textbf{0} &\hspace{0.2cm} \begin{matrix}
\begin{matrix}
0 & b\\
c & 0
\end{matrix}

   		\end{matrix}  & \textbf{0}  \\
\hline
\textbf{0} & \textbf{0} & I_{n-i-1}
\end{array} \right),\\
\zeta_6(\rho_i^{0}) =\left( \begin{array}{c|@{}c|c@{}}
   \begin{matrix}
     I_{i-1} 
   \end{matrix} 
      & \textbf{0} & \textbf{0} \\
      \hline
    \textbf{0} &\hspace{0.2cm} \begin{matrix}
   		\begin{matrix}
0 & x \\
\frac{1}{x} & 0
\end{matrix}
   		\end{matrix}  & \textbf{0}  \\
\hline
\textbf{0} & \textbf{0} & I_{n-i-1}
\end{array} \right),$$
and
$$\zeta_6(\rho_i^{1}) =\left( \begin{array}{c|@{}c|c@{}}
   \begin{matrix}
     I_{i-1} 
   \end{matrix} 
      & \textbf{0} & \textbf{0} \\
      \hline
    \textbf{0} &\hspace{0.2cm} \begin{matrix}
   		\begin{matrix}
0 & q \\
\frac{1}{q} & 0
\end{matrix}
   		\end{matrix}  & \textbf{0}  \\
\hline
\textbf{0} & \textbf{0} & I_{n-i-1}
\end{array} \right),$$
where $b,c,x \in \mathbb{C}$, $bc\neq0$,$x\neq0$ and $q\neq0$.

\item[(7)] $\zeta_7: M_2WB_n \rightarrow \text{GL}_n(\mathbb{C})$ such that
 $$\zeta_7(\sigma_i) =\left( \begin{array}{c|@{}c|c@{}}
   \begin{matrix}
     I_{i-1} 
   \end{matrix} 
      & \textbf{0} & \textbf{0} \\
      \hline
    \textbf{0} &\hspace{0.2cm} \begin{matrix}
\begin{matrix}
0 & b\\
\frac{1-d}{b} & d
\end{matrix}

   		\end{matrix}  & \textbf{0}  \\
\hline
\textbf{0} & \textbf{0} & I_{n-i-1}
\end{array} \right),\
\zeta_7(\rho_i^{0}) =\left( \begin{array}{c|@{}c|c@{}}
   \begin{matrix}
     I_{i-1} 
   \end{matrix} 
      & \textbf{0} & \textbf{0} \\
      \hline
    \textbf{0} &\hspace{0.2cm} \begin{matrix}
   		\begin{matrix}
0 & b \\
\frac{1}{b} & 0
\end{matrix}
   		\end{matrix}  & \textbf{0}  \\
\hline
\textbf{0} & \textbf{0} & I_{n-i-1}
\end{array} \right),$$
and
$$\zeta_7(\rho_i^{1}) =\left( \begin{array}{c|@{}c|c@{}}
   \begin{matrix}
     I_{i-1} 
   \end{matrix} 
      & \textbf{0} & \textbf{0} \\
      \hline
    \textbf{0} &\hspace{0.2cm} \begin{matrix}
   		\begin{matrix}
0 & b \\
\frac{1}{b} & 0
\end{matrix}
   		\end{matrix}  & \textbf{0}  \\
\hline
\textbf{0} & \textbf{0} & I_{n-i-1}
\end{array} \right),$$
where $b,d \in \mathbb{C}$,$d\neq1$, and $b\neq0$.
\end{itemize}
In the case of \( M_3W B_n \), only the three new relations (4.17)--(4.19), together with the relations of \( M_3V B_n \) \((3.51)-(3.61)\), generate the resulting set of new equations.
\begin{align}
  \sigma_1 \sigma_2 \rho_1^{0} & =  \rho_2^{0} \sigma_1 \sigma_2,\\
    \sigma_1 \sigma_2 \rho_1^{1} & = \rho_2^{1} \sigma_1 \sigma_2,\\ 
    \sigma_1 \sigma_2 \rho_1^{2} & = \rho_2^{2} \sigma_1 \sigma_2.
\end{align} which  gives twenty one new equations (4.3)-(4.16) and (4.20)-(4.26),
\begin{tcolorbox}[
  title=System of Equations Corresponding to Relations (4.19),
  colframe=black,
  colback=white,
  sharp corners,
  fonttitle=\bfseries,
  ]
\large  
\textbf{\boldmath (i) $\sigma_1 \sigma_2 \rho_1^{2} = \rho_2^{2} \sigma_1 \sigma_2$}
\begin{align}
  af + abh &= a, \\
ag + abi &= ab, \\
cf + adh &= cf, \\
cg + adi &= cg + adf, \\
dg + bdf & = bd, \\
ci + adh & = ci, \\
di + bdh & = d.  
\end{align}
\end{tcolorbox}
Solving this system of equations yields thirteen distinct solutions, each of which corresponds to a representation of the group $M_3WB_n$.
Similarly, for the general case of $M_kWB_n$ by solving the analogous set of equations along with the equations of $M_kVB_n$, we obtain all $3\cdot2^{k-1} + 1$ representation of $M_kWB_n$.
\end{proof}
To conclude this section, we propose the following question for further examination.
\begin{question}
What are the possible forms of the complex homogeneous $2$-local representations of multi-unrestricted braid group $M_kUB_n$ and there properties like faithfulness and Irreducibility?
\end{question}

\section{Extending LKBR representation of $B_n$ to $M_kWB_n$}

In this section, we consider a representation of $B_n$ which is not local, namely Lawrence-Krammer-Bigelow (LKB) representation. We aim to extend this representation to $M_kWB_n$ in some cases, in order to construct non-local representations of $M_kWB_n$. In the following, we set the explicit definition of the LKB representation.
 
\begin{definition} \cite{Law1990,Kram2002,Big2001}. \label{defLaw}
Let $V$ be a free $R$-module with basis $\{x_{i,j}, {1 \leq i < j \leq n}\}$, where $R=\mathbb{Z}[t^{\pm 1},q^{\pm 1}]$, the ring of Laurent polynomials on two variables $q$ and $t$. The LKB representation, $\mathcal{K}: B_n \rightarrow \text{GL}_{\frac{n(n-1)}{2}}(\mathbb{Z}[t^{\pm 1},q^{\pm 1}])$, for $n\geq 3$, is given by acting on the generators $\sigma_k, 1\leq k \leq n-1$, as follows.\\

$\mathcal{K}(\sigma_k)(x_{i,j})= \left\{\begin{array}{l}
tq^2x_{k,k+1}, \hspace{3.75cm} i=k<k+1=j, \\
(1-q)x_{i,k}+qx_{i,k+1}, \hspace{2.05cm} i<k=j,\\
x_{i,k}+tq^{k-i+1}(q-1)x_{k,k+1}, \hspace{1cm} i<k<k+1=j,\\
tq(q-1)x_{k,k+1}+qx_{k+1,j}, \hspace{1.3cm} i=k<k+1<j,\\
x_{k,j}+(1-q)x_{k+1,j}, \hspace{2.15cm} k<i=k+1<j,\\
x_{i,j}, \hspace{4.6cm} i<j<k \hspace{0.15cm} or \hspace{0.15cm} k+1<i<j,\\
x_{i,j}+tq^{k-i}(q-1)^2x_{k,k+1}, \hspace{1.2cm} i<k<k+1<j.
\end{array}\right.$
\end{definition}

\vspace{0.05cm}

We now introduce an extension of LKB representation of $B_n$ to $WB_n$ by letting $t=1$. This representation was given by V. Bardakov and P. Bellingeri in \cite{B2003}.

\begin{definition} \cite{B2003}
Let $V$ be a free $R$-module with basis $\{x_{i,j}\}, 1\leq i< j\leq n$, where $R=\mathbb{Z}[q^{\pm1}]$, the ring of Laurent polynomials on one variable $q$. We introduce the representation $\tilde{\mathcal{K}}:WB_n \rightarrow \text{GL}_\frac{n(n-1)}{2}(V)$, $n\geq 3,$ by the actions of $\sigma_i's$ and $\rho_i's$, $i=1, \ldots, n-1,$ on the basis of $V$ as follows.

\vspace{0.1cm}
$\left\{\begin{array}{l}
\sigma_i(x_{k,i})=(1-q)x_{k,i}+qx_{k,{i+1}}+q(q-1)x_{i,i+1},\\
\sigma_i(x_{k,{i+1}})=x_{k,i}$, \hspace{0.5cm} $k<i,\\
\sigma_i(x_{i,{i+1}})=q^2x_{i,{i+1}},\\
\sigma_i(x_{i,l})=q(q-1)x_{i,{i+1}}+(1-q)x_{i,l}+qx_{i+1,l},$ \hspace{0.5cm}$ i+1<l,\\
\sigma_i(x_{i+1,l})=x_{i,l},\\
\sigma_i(x_{k,l})=x_{k,l}$, \hspace{0.5cm} $\{ k,l\} \cap \{i,i+1\}=\emptyset, \\
\rho_i(x_{k,i})=x_{k,{i+1}},\\
\rho_i(x_{k,{i+1}})=x_{k,i}$, \hspace{0.5cm} $k<i,\\
\rho_i(x_{i,{i+1}})=x_{i,{i+1}},\\
\rho_i(x_{i,l})=x_{i+1,l},$ \hspace{0.5cm}$ i+1<l,\\
\rho_i(x_{i+1,l})=x_{i,l},\\
\rho_i(x_{k,l})=x_{k,l}$, \hspace{0.5cm} $\{ k,l\} \cap \{i,i+1\}=\emptyset. \\
\end{array} \right.$ 
\end{definition}

As the previous representation is complicated, we focus in our work for $n=3$. The representation $\tilde{\mathcal{K}}$ defined by the actions of the generators $\sigma_1, \sigma_2$, $\rho_1$, and $\rho_2$  of $WB_3$ on the standard basis $\{x_{1,2}, x_{1,3}, x_{2,3}\}$ as follows:
\[
\sigma_1 \to \begin{cases}
x_{1,2} \to q^2 x_{1,2} \\
x_{1,3} \to q(q - 1)x_{1,2} + (1 - q)x_{1,3} + q x_{2,3} \\
x_{2,3} \to x_{1,3}
\end{cases}, \]
\[
\sigma_2 \to \begin{cases}
x_{1,2} \to (1 - q)x_{1,2} + q x_{1,3} + q(q - 1)x_{2,3} \\
x_{1,3} \to x_{1,2} \\
x_{2,3} \to q^2 x_{2,3}
\end{cases},
\]
\[
\rho_1 \to \begin{cases}
x_{1,2} \to x_{1,2}\\
x_{1,3} \to x_{2,3} \\
x_{2,3} \to x_{1,3}
\end{cases}, \ \ and \ \ 
\rho_2 \to \begin{cases}
x_{1,2} \to x_{1,3}\\
x_{1,3} \to x_{1,2} \\
x_{2,3} \to x_{2,3}
\end{cases}.
\]
Thus, we can see that

$$\hat{\mathcal{K}}(\sigma_1)=
\left( \begin{array}{@{}c@{}}
\begin{matrix}
   		q^2 & q(q-1) & 0 \\
    	0 & 1-q & 1 \\
        0 & q & 0 \\
\end{matrix}
\end{array} \right), \hspace{0.5cm}
\hat{\mathcal{K}}(\sigma_2)=
\left( \begin{array}{@{}c@{}}
\begin{matrix}
   		1-q & 1 & 0 \\
    	q & 0 & 0 \\
        q(q-1) & 0 & q^2 \\
\end{matrix}
\end{array} \right),$$

$$\hat{\mathcal{K}}(\rho_1)=
\left( \begin{array}{@{}c@{}}
\begin{matrix}
   		1 & 0 & 0 \\
    	0 & 0 & 1 \\
        0 & 1 & 0 \\
\end{matrix}
\end{array} \right), \ \text{and} \ \
\hat{\mathcal{K}}(\rho_2)=
\left( \begin{array}{@{}c@{}}
\begin{matrix}
   		0 & 1 & 0 \\
    	1 & 0 & 0 \\
        0 & 0 & 1 \\
\end{matrix}
\end{array} \right).$$


\begin{proposition} \label{ppp}
Consider the mapping $\hat{\mathcal{K}}:  M_2VB_3 \rightarrow \text{GL}_3(\mathbb{Z}[q^{\pm 1}])$ defined by acting on the generators of $M_2VB_3$ as follows.

$$\hat{\mathcal{K}}(\sigma_1)=
\left( \begin{array}{@{}c@{}}
\begin{matrix}
   		q^2 & q(q-1) & 0 \\
    	0 & 1-q & 1 \\
        0 & q & 0 \\
\end{matrix}
\end{array} \right), \hspace{0.5cm}
\hat{\mathcal{K}}(\sigma_2)=
\left( \begin{array}{@{}c@{}}
\begin{matrix}
   		1-q & 1 & 0 \\
    	q & 0 & 0 \\
        q(q-1) & 0 & q^2 \\
\end{matrix}
\end{array} \right),$$

$$\hat{\mathcal{K}}(\rho_1^0)=
\left( \begin{array}{@{}c@{}}
\begin{matrix}
   		1 & 0 & 0 \\
    	0 & 0 & 1 \\
        0 & 1 & 0 \\
\end{matrix}
\end{array} \right), \hspace{0.5cm}
\hat{\mathcal{K}}(\rho_2^0)=
\left( \begin{array}{@{}c@{}}
\begin{matrix}
   		0 & 1 & 0 \\
    	1 & 0 & 0 \\
        0 & 0 & 1 \\
\end{matrix}
\end{array} \right),$$
$$\hat{\mathcal{K}}(\rho_1^1)=
\left( \begin{array}{@{}c@{}}
\begin{matrix}
   	0 & b & 0 \\
    	\frac{1}{b} & 0 & 0\\
        0 & 0 & 1 \\
\end{matrix}
\end{array} \right), \text{ and }
\hat{\mathcal{K}}(\rho_2^1)=
\left( \begin{array}{@{}c@{}}
\begin{matrix}
   	1 & 0 & 0 \\
    	0 & 0 & b \\
        0 & \frac{1}b & 0 \\
\end{matrix}
\end{array} \right).$$

Then, $\hat{\mathcal{K}}$ is a representation of $M_2WB_3$.
\end{proposition}
\begin{proof}
Direct computations imply that the relations of $M_2WB_3$ are satisfied for this mapping.
\end{proof}

Now, the following results are regarding the irreducibility and the faithfulness of the representation $\hat{\mathcal{K}}$ of $M_2WB_3$ in the next two theorems.

\begin{theorem}
The representation $\hat{\mathcal{K}}:  M_2WB_3 \rightarrow \text{GL}_3(\mathbb{Z}[q^{\pm 1}])$ given in Proposition \ref{ppp} is irreducible.
\end{theorem}

\begin{proof}
    It has been proven in \cite{Lev2010} that LKB representation of $B_n$ is irreducible. Therefore, $\hat{\mathcal{K}}$ is irreducible by Lemma \ref{llem}.
\end{proof}

\begin{theorem}
The representation $\hat{\mathcal{K}}:  M_2WB_3 \rightarrow \text{GL}_3(\mathbb{Z}[q^{\pm 1}])$ given in Proposition \ref{ppp} is unfaithful.
\end{theorem}

\begin{proof}
In \cite{123456}, it has been proven that the restriction of the representation $\hat{\mathcal{K}}$ to $WB_3$ is unfaithful. Thus, $\hat{\mathcal{K}}$ is unfaithful.
\end{proof}

\section{Conclusion}
In this paper, we have discussed all complex homogeneous $2$-local representations of the multi-virtual braid group $M_kVB_n$ and the multi-welded braid group $M_kWB_n$. We examined key properties of these representations, such as \textbf{faithfulness} and \textbf{irreducibility}, which are fundamental to understanding their algebraic and geometric structures. In addition, we have constructed non-local representations of $M_2WB_3$, making a path towards finding more such representations in future work.

\vspace{0.1cm}

This study also opens up several interesting directions for future research. One such direction is the classification of complex homogeneous $2$-local representations of the \textbf{multi-unrestricted braid group} $M_kUB_n$. Key open problems include determining how many such representations exist, what their structural characteristics are, and whether they satisfy desirable properties such as faithfulness and irreducibility.

\vspace{0.1cm}

These questions highlight the richness of the subject and the potential for further exploration in the field of braid group representations.

\section{Acknowledgment}
The first author acknowledges the support of the University Grants Commission
(UGC), India, for a research fellowship with NTA Ref. No.231610035955 . The cor-
responding author acknowledges the support of the Anusandhan National Research
Foundation (ANRF) with sanction order no. CRG/2023/004921.



\vspace{0.2cm}

\begin{thebibliography}{999}

\bibitem{artin1946} E. Artin, \emph{Theorie der z{\"o}pfe}, {Abhandlungen Hamburg}, 4, (1925), 47-72.

\bibitem{B2003} V. Bardaov and P. Bellingeri, {\it The structure of the group of conjugating automorphisms and the linear representation of the braid groups of some manifolds}, {Algebra i Logika.}, 42 (5), (2003), 515-541.

\bibitem{19} V. Bardakov and P. Bellingeri, {\it On representation of braids as automorphisms of free groups and corresponding linear representations}, {Knot Theory and Its Applications, Contemp. Math.}, 670, Amer. Math. Soc., Providence (2016), 285-298.

\bibitem{up} V. Bardakov, T. kozlovskaya, K. Negi, and M. Prabhakar, \emph{Multi-Virtual Braid Group}, {Under Preperation}.

\bibitem{Big1999} S. Bigelow, {\it The Burau Representation is not faithful for $n=5$}, { Geometry \& Topology}, 3, (1999), 397-404.

\bibitem{Big2001} S. Bigelow, {\it Braid groups are linear}, {J. Amer. Math. Soc.}, 14, (2000), 471-486.

\bibitem{Bir1975} J. Birman, Braids, links and mapping class groups, {\it Annals of Mathematical studies}, Princeton University Press, 82, (1974).

\bibitem{Bur1936} W. Burau, {\it Braids, Uber Zopfgruppen and gleichsinnig verdrillte Verkettungen}, {Abh. Math. Semin. Hamburg Univ}, 11, (1936), 179-186.

\bibitem{M.Ch} M. Chreif and M. Dally, \emph{On the irreducibility of local representations of the braid group $B_n$}, {Arab. J. Math.}, 13, (2024), 263-273.

\bibitem {For1996} E. Formanek, {\it Braid group representations of low degree}, { Proc. London Math. Soc.}, 73, (1996), 279-322.

\bibitem{KauffmanLambropoulou2004} L. Kauffman and S.Lambropoulou, {\it Virtual Braids}, {Fund. Math.}, 184, (2004), 159-186.

\bibitem{LH2024} L. Kauffman, {\it Multi-Virtual Knot Theory}, (2024), arXiv:2409.07499 [math.GT].

\bibitem{Kram2002} D. Krammer, {\it Braid groups are linear}, {Annals Math.}, 155, (2002), 131-156.

\bibitem{Law1990} R. Lawrence, {\it Homological representations of the Hecke algebra}, { Comm. Math. Phys.}, 135, (1990), 141–191.

\bibitem{Lev2010} C. Levaillant and D. Wales, {\it Parameters for which the Lawrence–Krammer representation is reducible}, {Journal of Algebra}, 323, (2010), 1966-1982.

\bibitem{Long1992} D. Long and M. Paton, \emph{The Burau representation of the braid group $B_n$ is not faithful for $n\geq 6$}, {Topology}, 32, (1992) 439-447.

\bibitem{Mayassi2025} T. Mayassi and M. Nasser, {\it Classification of homogeneous local representations of the singular braid monoid}, (2025), arXiv:submit/6150117.

\bibitem{Mik2013} Y. Mikhalchishina, {\it Local representations of braid groups}, {Sib. Math. J}, 54, (2013), 666-678.

\bibitem{Moo1991} J. Moody, {\it The Burau representation of the braid group $B_n$ is not faithful for large $n$}, {Bull. Amer. Math. Soc.}, 25, (1991), 379-384.

\bibitem{123456} M. Nasser, \emph{The faithfulness of an extension of Lawrence-Krammer-Bigelow representation on the group of conjugating automorphisms $C_n$ in the cases $n=3$ and $n=4$}, {BAU Journal - Science and Technology}, 6 (1), (2024), Article 1.

\bibitem{55} M. Nasser, \emph{Necessary and sufficient conditions for the irreducibility of a linear representation of the braid group $B_n$}, {Arab. J. Math.}, 13, (2024), 333-339.

\bibitem{MNas2025} M. Nasser, M. Chreif, and M. Dally, {\it Local representations of the flat virtual braid group}, (2025), arXiv:2503.06607v1.

\bibitem{Mnas2025} M. Nasser, V. Keshari, and M. Prabhakar, {\it Matrix representation of the twisted virtual braid group and its extensions}, (2025), arXiv:2506.03806v1.

\bibitem{Vinberg} E. Vinberg, Linear Representations of Groups, Translated from the Russian by A. Iacob (Birkhäuser, 1989).

\end{thebibliography}
\end{document}